\documentclass[11pt,reqno]{amsart}

\usepackage{amsfonts,amsmath,amssymb, amsthm,mathtools}
\usepackage[usenames]{color}
\newtheorem{theorem}{Theorem}[section]
\newtheorem{proposition}[theorem]{Proposition}
\newtheorem{lemma}[theorem]{Lemma}
\newtheorem{corollary}[theorem]{Corollary}
\newtheorem{remark}[theorem]{Remark}
\newtheorem{definition}[theorem]{Definition}
\newtheorem{example}[theorem]{Example}
\newcommand{\be}{\begin{equation}}
\newcommand{\ee}{\end{equation}}

\numberwithin{equation}{section}

\newcommand{\sfd}{{\sf d}}

\DeclareMathOperator{\sing}{Sing}

\DeclareMathOperator{\Div}{div}

\DeclareMathOperator{\V}{\rm{Vol}}
\DeclareMathOperator{\Vol}{\rm{Vol}}
\DeclareMathOperator{\Pe}{\mathcal{P}}
\DeclareMathOperator{\dist}{dist}

\DeclareMathOperator{\supp}{spt}

\DeclareMathOperator{\A}{\mathcal{A}}
\DeclareMathOperator{\Ha}{\mathcal{H}}

\DeclareMathOperator{\rad}{rad}
\renewcommand{\div}{{\text {div}}}

\newcommand{\R}{\mathbb{R}}
\newcommand{\eps}{{\varepsilon}}
\newcommand{\de}{\partial}

\renewcommand{\dim}{{\text{dim}}}

\renewcommand{\d}{{\textrm {d}}}

\newcommand\sdif{{\vartriangle}}

\renewcommand{\P}{{\mathcal{P}}}

\newcommand\N{{\mathbb N}}

\newcommand\cH{{\mathcal H}}

\newcommand\cL{{\mathcal L}}

\newcommand{\Id}{{\text{Id}}}

\newcommand{\fn}{\normalcolor}

\begin{document}

\title{On a isoperimetric-isodiametric inequality}

\author{Andrea Mondino}  \thanks{A. Mondino,  MSRI-Berkeley\&Universit\"at Z\"urich, Institut f\"ur Mathematik. email: \textsf{andrea.mondino@math.uzh.ch}}
\author {Emanuele Spadaro} \thanks{E. Spadaro, Max-Planck-Institut,  Institut f\"ur Mathematik.   email: \textsf{spadaro@mis.mpg.de}}

\begin{abstract} 
The Euclidean mixed isoperimetric-isodiametric inequality states that the round ball maximizes the volume under 
constraint on the product between boundary area and radius. The goal of the paper is to investigate such mixed 
isoperimetric-isodiametric inequalities in Riemannian manifolds. We first prove that the same inequality, with the 
sharp Euclidean constants,  holds on Cartan-Hadamard spaces as well as on minimal submanifolds of $\R^n$. The equality 
cases are also studied and completely characterized; in particular, the latter gives a  new link with  free boundary 
minimal submanifolds in a Euclidean ball. We also consider the case of manifolds with non-negative Ricci curvature and 
prove a new comparison result stating that metric balls in the manifold have product of  boundary area and radius  
bounded by the Euclidean counterpart and equality holds if and only if the ball is actually Euclidean. 
\\We then pass to consider the problem of the existence of optimal shapes (i.e. regions minimizing the product of  
boundary area and radius under the constraint of having fixed enclosed volume), called here isoperimetric-isodiametric 
regions.
While it is not difficult to show existence if the ambient manifold is compact, the situation changes dramatically if 
the manifold is not compact: indeed we give examples of spaces where there exists no isoperimetric-isodiametric region 
(e.g. minimal surfaces with planar ends and more generally $C^0$-locally-asymptotic Euclidean Cartan-Hadamard 
manifolds), and we prove that on the other hand on $C^0$-locally-asymptotic  Euclidean manifolds with non-negative 
Ricci curvature there exists an isoperimetric-isodiametric region for every positive volume (this class of spaces 
includes a large family of metrics playing a key role in general relativity and Ricci flow: the so called Hawking 
gravitational instantons and the Bryant-type Ricci solitons).
\\Finally we pass to prove the optimal regularity of the boundary of  isoperimetric-isodiametric regions: in the part which does not touch a minimal enclosing ball the boundary is  a smooth  hypersurface outside of a closed subset of Hausdorff co-dimension $8$,  and in a neighborhood of the contact  region the boundary is a $C^{1,1}$-hypersurface with explicit estimates on the $L^\infty$-norm of the mean curvature.
\end{abstract}

\maketitle

\tableofcontents

\section{Introduction}
One of the oldest questions of mathematics is the \emph{isoperimetric} problem: \emph{What is the largest amount of volume that can be
enclosed by a given amount of area?} 
A related classical question is the \emph{isodiametric} problem: \emph{What is the largest amount of volume that can be
enclosed by a domain having a fixed diameter?}
 \\
 
In this paper we address a mix of the previous two questions, namely we investigate the following mixed isoperimetric-isodiametric problem:  \emph{What is the largest amount of volume that can be
enclosed by a domain having a fixed product of  diameter and boundary area?}
\\

Of course, if we ask  the three above questions in the Euclidean space, the answer is given by the round balls  of 
the suitable radius; but, of course, the situation in non-flat geometries is much more subtle.   We start by recalling 
classical material on the isoperimetric problem which motivated our investigation on the  mixed 
isoperimetric-isodiametric one. 
\\

The  solution of the isoperimetric problem in the Euclidean space $\R^n$ can be summarized by the classical 
isoperimetric inequality
\begin{equation}\label{eq:IsopIne}
n\, \omega_{n}^{\frac{1}{n}}\V(\Omega)^{\frac{n-1}{n}} \leq \A(\partial \Omega) \, , \quad \text{for every } \Omega 
\subset \R^n \text{ open subset with smooth boundary,} \quad
\end{equation}
where  $\V(\Omega)$ is the $n$-dimensional Hausdorff measure of $\Omega$ (i.e. the ``volume'' of $\Omega$),  
$\A(\partial \Omega)$ is the $(n-1)$-dimensional Hausdorff measure of $\partial \Omega$ (i.e. the ``area'' of $\partial 
\Omega$), and   $\omega_n:=\V(B^n)$ is the volume of the unit ball in $\R^n$.  As it is well known, the regularity 
assumption  on $\Omega$ can be relaxed a lot (for instance \eqref{eq:IsopIne} holds for every set $\Omega$ of finite 
perimeter) but let us not enter in technicalities  here since we are just motivating our problem.

As anticipated above, in the present paper we will not deal with the isoperimetric problem itself but we will focus on 
a mixed isoperimetric-isodiametric problem. Let us start by stating the Euclidean mixed isoperimetric-isodiametric 
inequality which will act as model for this paper. Given  a bounded  open subset $\Omega \subset \R^n$  with smooth 
boundary, by the divergence theorem in $\R^n$ (see Section \ref{sec:Eucl} for the easy proof),  we have
\begin{equation}\label{eq:IsoPDRn}
n \V(\Omega) \leq \rad(\Omega) \A(\partial \Omega),
\end{equation} 
where $\rad(\Omega)$ is the radius of the smallest ball of $\R^n$ containing $\Omega$ (see \eqref{eq:Defr} for the 
precise definition). As observed in Remark \ref{rem:rigidityEuc}, inequality \eqref{eq:IsoPDRn} is \emph{sharp} and 
\emph{rigid}; indeed, equality occurs if and only if $\Omega$ is a round ball in $\R^n$.

In sharp contrast  with the classical isoperimetric problem, where both problems are still open in the general case, it 
is not difficult to show  that the inequality  \eqref{eq:IsoPDRn} holds in 
Cartan-Hadamard spaces (i.e.~simply connected Riemannian manifolds with non-positive sectional curvature) and on 
minimal submanifolds of $\R^n$, see Proposition \ref{thm:CH}, Proposition \ref{thm:IsoPDMinExt} and Proposition 
\ref{thm:IsoPDMinint}. Even if the validity of inequality \eqref{eq:IsoPDRn} in such spaces is probably known to 
experts, we included it here in order to motivate the reader and also because the equality case for minimal submanifolds 
present an interesting link with free-boundary minimal surfaces: equality is attained in  
\eqref{eq:IsoPDRn} if and only if the minimal submanifold is a free boundary minimal surface in a Euclidean ball (see  
Proposition \ref{thm:IsoPDMinExt} for the precise statement and Remarks  \ref{rem:FreeBCritMet}-\ref{rem:ExFreeBounday} 
for more information about free boundary minimal surfaces).  

If on one hand the negative curvature gives a stronger isoperimetric-isodiametric inequality, on the the other hand we show that non-negative Ricci curvature forces metric balls to  satisfy a weaker isoperimetric-isodiametric inequality. The precise statement is the following.

\begin{theorem}[Theorem \ref{thm:Ric>0}] \label{thm:Ric>0Intro}
Let $(M^n,g)$ be a complete (possibly non compact) Riemannian $n$-manifold with non-negative Ricci curvature.  Let  $B_r\subset M$ be a metric ball  of volume $V=\V_g(B_r)$, and denote with $B^{\R^n}(V)$ the round ball in $\R^n$ having volume $V$. Then
\begin{equation}
\rad(B_r) \, \A(\de B_r)=r \, \A(\de B_r) \leq n \V_g(B_r)=\rad_{\R^n}(B^{\R^n}(V)) \, \A_{\R^n} (\de B^{\R^n}(V)).
\end{equation}
Moreover equality holds if and only if $B_r$ is \emph{isometric} to a round ball in the Euclidean space $\R^n$. In 
particular, for every $V\in (0,\V_g(M))$ it holds
\begin{equation}\label{eq:CompRic>0Intro}
\inf \{ \rad(\Omega) \P(\Omega) \,:\,  \Omega \subset M, \, \V_g(\Omega)=V \} \leq n V =  \inf \{ \rad(\Omega) \P(\Omega) \,:\,  \Omega \subset \R^n, \, \V_{\R^n}(\Omega)=V \}, 
\end{equation}
with equality for some $V\in (0,\V_g(M))$ if and only if every metric ball in $M$ of volume $V$ is isometric to a round ball in $\R^n$. In particular if equality occurs for some $V\in (0,\V_g(M))$ then  $(M,g)$ is flat, i.e. it has identically zero sectional curvature.
\end{theorem}

\begin{remark}
\rm{Since by Bishop-Gromov volume comparison we know that if $Ric_g\geq 0$ then for every metric ball 
$B_r(x_0)\subset M$ it holds  $\V_g(B_r(x_0))\leq \omega_n r^n= \V_{\R^n}(B^{\R^n}_r)$, it follows that
$$\rad(B_r(x_0)) \geq \rad_{\R^n} (B^{\R^n}(V)),$$  
where $B^{\R^n}(V)$ is a Euclidean ball of volume $V=\Vol_g(B_r(x_0))$. Therefore Theorem \ref{thm:Ric>0Intro} in 
particular implies  that $\P(B_r(x_0))\leq \P_{\R^n}(B^{\R^n}(V))$, but is a strictly stronger statement which at  best 
of our knowledge is original.
The aforementioned  counterpart of Theorem \ref{thm:Ric>0Intro} for the isoperimetric problem was proved instead by 
Morgan-Johnson \cite[Theorem 3.5]{MJ} for compact manifolds and extended to non-compact manifolds in \cite[Proposition 
3.2]{MoNa}.} \hfill $\Box$
\end{remark}

In Section \ref{sec:Existence} we investigate the existence of optimal shapes in  a general Riemannian manifold $(M,g)$.
More precisely, given a   measurable subset $E\subset M$ we denote with $\P(E)$ its perimeter and define its extrinsic radius as
\[
\rad(E) := \inf\left\{r>0 \;:\; \V_g(E\setminus B_r(z_0)) = 0 \;\text{for some } z_0 \in M\right\},
\]
where $B_r(z_0)$ denotes the open metric ball with center $z_0$ and radius $r>0$. 
We consider the following minimization problem: for every fixed $V \in (0, \V_g(M))$, 
\begin{equation}\label{e:min problemIntro}
\min \; \Big\{ \rad(E)\,\P(E) \;:\; E \subset M,\;\V_g(E) = V\Big\},
\end{equation}
and call the minimizers of \eqref{e:min problemIntro} isoperimetric-isodiametric sets (or regions). To best of our 
knowledge this is first time such a problem is considered in literature.
\\As it happens also for the isoperimetric problem, we will find that if the ambient manifold is compact then for 
every volume there exists an isoperimetric-isodiametric region (see Theorem \ref{thm:Exist} and Corollary 
\ref{cor:ExCpt}) but if the ambient space is non-compact the situation changes dramatically.  Indeed in Examples 
\ref{ex:NoExMinimal}-\ref{ex:NoExALE} we show that in complete minimal submanifolds with planar ends (like the helicoid)  
and in asymptotically locally Euclidean Cartan-Hadamard manifolds there exists  no isoperimetric-isodiametric region of 
positive volume. On the other hand, we show that in  $C^0$-locally asymptotically Euclidean manifolds (see Definition 
\ref{def:C0Conv} for the precise notion) with non negative Ricci curvature for every volume there exists an 
isoperimetric-isodiametric region:
 
\begin{theorem}[Theorem \ref{thm:ExistRic>0}] \label{thm:ExistRic>0Intro}
Let $(M,g)$ be a complete Riemannian $n$-manifold with non-negative Ricci curvature and  fix any reference point  
$\bar{x}\in M$. Assume that for any diverging sequence of points  $(x_k)_{k \in N} \subset M$, i.e. $\sfd(x_k, 
\bar{x})\to \infty$, the sequence of pointed manifolds $(M, g, x_k)$ converges in the pointed $C^{0}$-topology to the 
Euclidean space $(\R^n, g_{\R^n}, 0)$. 
\\Then for every $V \in (0, \V_g(M))$ there exists a minimizer of the problem   \eqref{e:min problemIntro}, in other 
words there exists an isoperimetric-isodiametric region of volume $V$.
\end{theorem}

Let us mention that the counterpart of Theorem \ref{thm:ExistRic>0Intro} for the isoperimetric problem was proved in \cite{MoNa} capitalizing on the work  by Nardulli \cite{NardAJM}.

\begin{remark}
\rm{
It is well known that the only manifold  with non-negative Ricci curvature and $C^0$-\emph{globally} asymptotic to 
$\R^n$ is $\R^n$ itself. Indeed if $M$ is   $C^0$-globally asymptotic to $\R^n$ then 
$$\lim_{R\to \infty} \frac{\V_g(B_R(\bar{x}))}{\omega_n R^n}=1,$$ which by the rigidity statement associated to the Bishop-Gromov inequality implies that $(M,g)$ is globally  isometric to $\R^n$.  On the other hand, the assumption of  Theorem \ref{thm:ExistRic>0Intro} is much weaker as it ask $(M,g)$ to be just \emph{locally} asymptotic to $\R^n$ in $C^0$ topology and many important examples enter in this framework as explained in next Example \ref{Ex:ALE}.}
 \hfill$\Box$
\end{remark}

\begin{example}\label{Ex:ALE}
\rm{
The class of manifolds satisfying the assumptions of Theorem  \ref{thm:ExistRic>0Intro} contains many geometrically and physically relevant examples.
\begin{itemize}
 \item \emph{Eguchi-Hanson and more generally ALE gravitational instantons.} These are 4-manifolds, solutions of the 
Einstein vacuum equations with null cosmological constant (i.e.~they are Ricci flat, $Ric_g \equiv 0$), they are 
non-compact with just one end which is topologically a quotient of $\R^4$ by a finite subgroup of $O(4)$, and the 
Riemannian metric $g$ on this end is asymptotic to the Euclidean metric up to terms of order $O(r^{-4})$,
$$g_{ij}=\delta_{ij}+O(r^{-4}),$$
with appropriate decay in the derivatives of $g_{ij}$ (in particular, such metrics are  $C^0$-locally asymptotic, in the sense of Definition \ref{def:C0Conv}, to the Euclidean $4$-dimensional space). 
The first example of such manifolds was discovered by Eguchi and Hanson in  \cite{EH78};  the authors, inspired by the 
discovery of self-dual instantons in Yang-Mills Theory, found a self-dual ALE instanton metric. 
The Eguchi-Hanson example was then generalized by Gibbons and Hawking \cite{GH}, see  also the work by Hitchin 
\cite{Hi1}.  These  metrics constitute the  building blocks of the Euclidean quantum gravity theory of Hawking  (see 
\cite{H1,H2}). The ALE Gravitational Instantons were classified in 1989 by Kronheimer (see \cite{Kro1,Kro2}).

 \item \emph{Bryant-type solitons}. The Bryant solitons, discovered by  R. Bryant  \cite{Bryant}, are special but 
fundamental solutions to the Ricci flow (see for instance the work of Brendle \cite{Bre3D,BreHD} for higher 
dimension). Such metrics are complete,   have non-negative Ricci curvature (they actually satisfy the stronger condition 
of having nonnegative curvature operator) and are locally $C^0$-asymptotically Euclidean. Other soliton examples fitting 
our assumptions are given by Catino-Mazzieri in \cite{CM12}.  \hfill$\Box$
 \end{itemize}}   
\end{example}
The last Section \ref{sec:OptReg} is then devoted to establish the optimal regularity for 
isoperimetric-isodiametric regions under suitable assumptions on regularity of the enclosing ball.  We first observe 
that outside of the contact region with the minimal enclosing ball $B$, such sets are locally minimizers of the 
perimeter under volume constraint. Therefore by classical results (see, for example, \cite[Corollary 
3.8]{Morgan}) in the interior of $B$ the boundary of the region is a smooth hypersurface outside a singular set of 
Hausdorff co-dimension at least $8$.

The rest of the paper is devoted to prove the optimal regularity at the contact region. We first   show in Section 
\ref{s:C1alpha} that isoperimetric-isodiametric  regions  are almost-minimizers for the perimeter (see Lemma \ref{l:almost 
min}) and therefore, by a result of Tamanini \cite{Tam} their boundaries are $C^{1,1/2}$ regular (see Proposition \ref{p:regularity}).  
In Section \ref{s:LinftyEst}, by means of geometric  comparisons and sharp first variation  arguments, we show that the 
mean curvature of the boundary of an isoperimetric-isodiametric  region is in $L^\infty$ with explicit estimates. 
Finally in Section \ref{s:Opt} we establish the optimal $C^{1,1}$ regularity. We mention that, strictly speaking, 
Section \ref{s:LinftyEst} is not needed to prove the optimal regularity; in any case we included such section since  provides an explicit 
sharp $L^\infty$-estimate on the mean curvature and is of independent interest. Now the let us state the main regularity result.

\begin{theorem}[Theorem \ref{p:regularity2}]\label{p:regularity2intro}
Let $E\subset M$ be an isoperimetric-isodiametric set and $x_0 \in M$ be such that $\V_g(E \setminus 
B_{\rad(E)}(x_0)) = 0$. Assume that $B:= B_{\rad(E)}(x_0))$ has smooth boundary.
Then, there exists $\delta>0$ such that $\de E \setminus B_{\rad(E) - \delta}(x_0)$ is $C^{1,1}$ regular.
\end{theorem}

An essential ingredient in the proof of Theorem  \ref{p:regularity2intro} is Proposition \ref{p:quadratic}, which roughly tells that the boundary of $E$ leaves the obstacle at most quadratically. Then the conclusion will follow by combining Schauder estimates outside of the contact region (see Lemma \ref{l:stima D2}) with the general fact that functions which leave the first order approximation quadratically are $C^{1,1}$ -- see Lemma \ref{l:distrib}. 
Although the techniques exploited for this part of the paper are inspired by the ones introduced in the study of the classical obstacle problem (cf., for example, \cite{Caffarelli}), here we treat the geometric case of the area functional in a Riemannian manifold with volume constraints and we take several short-cuts by thanks to some specifically geometric arguments, such as the theory of almost minimizers.
{ In particular, such geometric situation doesn't seem to be trivially covered by the regularity results for nonlinear variational inequalities, as developed for example by Gerhardt \cite{Gerhardt} -- see Remark~\ref{r:gerhardt}.}

\begin{remark}
\rm{
Note that the $C^{1,1}$ regularity is optimal, because in general one cannot expect to have continuity of the second 
fundamental form of $\de E$ across the free boundary of $\de E$, i.e.~the points on the relative (with respect to $\de 
B$) boundary of $\de E \cap \de B$. The same is indeed true for the simplest case of the classical obstacle problem.}
\end{remark}

\noindent {\bf Acknowledgment.}  Part  of the work has been  developed while A. M. was lecturer at the  Institut f\"ur Mathematik at the  Universit\"at Z\"urich and the project was finalized when A. M. was in residence at the Mathematical Science Research Institute in Berkeley, California, during the spring  semester 2016 and was  supported by the National Science Foundation under the Grant No. DMS-1440140. He wishes to express his gratitude to both the institutes for the stimulating atmosphere and the excellent working conditions.  

\section{Notation, Preliminaries and the Euclidean case}\label{sec:Eucl}
Let $(Z, \sfd)$ be a metric space. Given an open subset $\Omega\subset Z$, we define its \emph{extrinsic radius} as 
\be\label{eq:Defr}
\rad(\Omega):=\inf \{r>0: \; \Omega \subset B_r(z_0) \; \text{for some } z_0 \in Z \} \quad,
\ee
where $B_r(z_0)$ denotes the  open metric ball of center $z_0$ and  radius $r>0$. 

The model inequality for the first part of the paper is the Euclidean mixed isoperimetric-isodiametric  inequality obtained by the following integration by parts. Let $\Omega \subset \R^n$ be a bounded open subset with smooth boundary 
 and let $x_0 \in \R^n$  be a point such that 
\be\label{eq:x0max}
\max_{x\in \bar{\Omega}} |x-x_0|= \rad(\Omega).
\ee
Denoted with  $X$ the vector field  $X(x):=x-x_0$, by the divergence theorem in $\R^n$ we then get
\be\label{eq:IsoPDRn1}
n \V(\Omega)= \int_{\Omega} \Div X \, d \Ha^n= - \int_{\partial \Omega} X \cdot \nu  \, d{\mathcal{H}}^{n-1} \leq \rad(\Omega) \A(\partial \Omega),
\ee 
where $\V(\Omega)$ denotes the Euclidean $n$-dimensional  volume of $\Omega$, $\nu$ is the inward pointing unit normal vector and $\A(\partial \Omega)$ is the Euclidean $(n-1)$-dimensional area of $\partial \Omega$, which here is assumed to be smooth. Notice that, analogously, if $\Omega \subset \R^n$ is a finite perimeter set one gets the inequality
\be\label{eq:IsoPDRnPer}
\V(\Omega)\leq \frac{\rad(\Omega)}{n} \Pe(\Omega),
\ee  
where, of course, $\Pe(\Omega)$ denotes the perimeter of $\Omega$ (see \S~\ref{Sec:notation} for the definitions of $\Pe(\Omega)$ and $\rad(\Omega)$ for
finite perimeter sets).
\begin{remark} \label{rem:rigidityEuc}
\rm{The inequalities \eqref{eq:IsoPDRn1} and \eqref{eq:IsoPDRnPer} are \emph{sharp} and \emph{rigid}: indeed equality occurs if and only if  $\Omega$ is a round ball.}  \hfill$\Box$
 \end{remark}

 \section{Euclidean  isoperimetric-isodiametric inequality in Cartan-Hadamard manifolds and minimal submanifolds}\label{Sec:CH}
 In order to motivate and gently introduce the reader to the  topic, in this section we will prove that the Euclidean isoperimetric-isodiametric inequality holds with the same constant in Cartan-Hadamard spaces and in minimal submanifolds.  Possibly apart  from the rigidity statements, here we do not claim originality since such inequalities are probably well known to experts (cf. \cite{BuZa}, \cite{HoSp}, \cite{MiSi}). However we included this section for the following reasons:
 \begin{itemize}
 \item While for the  isoperimetric-isodiametric inequality the proofs are a consequence of a non difficult integration by parts argument,  the corresponding statements for the classical isoperimetric inequality are still open problems (see Remark  \ref{rem:OPCH} and Remark \ref{rem:OPMin}). This suggest that possibly also in other situations  isoperimetric-isodiametric inequalities  may behave better than the classical isoperimetric ones.
 \item  The rigidity statements, in case of minimal submanifolds, show interesting connections between the  isoperimetric-isodiametric inequality and free boundary minimal surfaces, a topic which recently has received a lot of attention (for more details see Remark \ref{rem:FreeBCritMet} and Remark \ref{rem:ExFreeBounday}).
 \end{itemize}

 \subsection{The case of Cartan-Hadamard manifolds}
 Recall that a Cartan-Hadamard $n$-manifold is a complete simply connected Riemannian $n$-dimensional manifold with non-positive sectional curvature. By a classical theorem of Cartan and Hadamard (see for instance \cite{DoC}) such manifolds are diffeomorphic to $\R^n$ via the exponential map. The next result  is a sharp and rigid mixed isoperimetric-isodiametric inequality in such spaces. For this section, without loosing much, the non-expert reader may assume the region $\Omega \subset M$ to have smooth boundary, in this case the perimeter is just the standard $(n-1)$-volume of the boundary (the perimeter will instead play a role in the next sections about existence and regularity of optimal sets).
  
 \begin{proposition}\label{thm:CH}
 Let $(M^n,g)$ be a Cartan-Hadamard manifold. Then for every smooth  open subset (or more generally for every finite perimeter set)
$\Omega \subset M^n$  it holds
\be\label{eq:IsoPDCH}
n \V(\Omega) \leq \rad(\Omega) \, \A(\partial \Omega)
\ee 
 where $\V(\Omega)$ denotes the $n$-dimensional Riemannian volume of $\Omega$ and  $\A(\partial \Omega)$ the $(n-1)$-dimensional area of the smooth boundary $\partial \Omega$ (in case $\Omega$ is a finite perimeter set, just replace 
 $\A(\partial \Omega)$ with $\Pe(\Omega)$, the perimeter of $\Omega$ in the right hand side, and $\rad(\Omega)$ is as in \S~\ref{Sec:notation}\footnote{For readers' convenience we recall here the definition of $\rad(\Omega)$ for a finite perimeter set $\Omega\subset M$: $\rad(\Omega):=\inf\{r>0\, :\, \V(\Omega\setminus B_r)=0, \; B_{r}\subset M \text{ metric ball}\}$.}).
Moreover if for some $\Omega$ the equality is achieved, then $\Omega$ is  isometric to  \emph{an Euclidean ball}.
 \end{proposition}
 
 \begin{proof}
 Let $\Omega\subset M^n$ be a subset with finite perimeter; without loss of generality we can assume that $\Omega$ is bounded (otherwise $\rad(\Omega)=+\infty$ and the inequality is trivial). Let $x_0 \in M^n$ be such that 
 $$\max_{x \in \bar{\Omega}} \sfd(x,x_0)=\rad(\Omega),$$
 where $\sfd$ is the Riemannian distance on $(M^n,g)$, for convenience we will also denote $\sfd_{x_0}(\cdot):= \sfd(x_0, \cdot)$.  
 Let $u:=\frac{1}{2} \sfd^2_{x_0}$; by the aforementioned Cartan-Hadamard Theorem (see for instance \cite{DoC}) we know that $u:M^n\to \R^+$ is smooth and by the Hessian comparison Theorem 
 one has $(D^2 u)_{ij}\geq g_{ij}$; in particular, by tracing, we get $\Delta u\geq n$.   Therefore, by the divergence theorem, we infer
 
 \begin{eqnarray}
 n \V(\Omega) &\leq& \int_{\Omega} \Delta u \, d\mu_g =- \int_{\partial^* \Omega} g(\nabla u, \nu) \, d \Ha^{n-1} = - \int_{\partial^* \Omega} \sfd(x, x_0) \; g(\nabla \sfd_{x_0}, \nu) \, d \Ha^{n-1} \nonumber \\
 &\leq&   \rad(\Omega) \Ha^{n-1}(\partial^* \Omega) = \rad(\Omega) \Pe(\Omega) \label{eq:proofCH1},
 \end{eqnarray}
where $\mu_g$ is the measure associated to the Riemannian volume form, $\partial^*\Omega$ is the reduced boundary of $\Omega$ (of course, in case $\Omega$ is a smooth open subset one has $\partial^* \Omega=\partial \Omega$), $\nu$ is the inward pointing unit normal vector (recall that it is $\Ha^{n-1}$-a.e. well defined on $\partial^* \Omega$), and we used that $\sfd_{x_0}$ is 1-Lipschitz. 
Of course \eqref{eq:proofCH1} implies  \eqref{eq:IsoPDCH}.
Notice that if equality holds in  the second line, then $\Omega$ is a metric ball of center $x_0$ and radius $\rad(\Omega)$.  Moreover if equality occurs in the first inequality of the first line then we must have $(D^2 \sfd^2_{x_0})_{ij}\equiv 2 g_{ij}$ on $\Omega$, and by standard comparison (see for instance  \cite[Section 4.1]{Rit}) it follows that $\Omega$ is flat. But since the exponential map in $M$ is a global diffeomorphism it follows that $\Omega$ is isometric to an Euclidean ball.
 \end{proof}


\begin{remark}[Euclidean isoperimetric inequality on Cartan-Hadamard spaces] \label{rem:OPCH}
\rm{The statement corresponding to Proposition \ref{thm:CH} for the isoperimetric problem is the following celebrated conjecture: 
 \emph{Let $(M^n, g)$ be a Cartan-Hadamard space, i.e. a complete simply connected Riemannian $n$-manifold with non-positive sectional curvature. Then every smooth open subset $\Omega \subset M^n$   satisfies  the Euclidean isoperimetric inequality}.
\\This conjecture is  generally attributed to Aubin \cite[Conj. 1]{Aub} but has its roots in  earlier work by Weil \cite{Weil}, as we are going to explain. The problem has been solved affirmatively in the following cases: in dimension 2 by Weil \cite{Weil} in 1926 (Beckenbach and Rad\'o \cite{BR} gave an independent proof in 1933, capitalizing on a result of Carleman \cite{Car} for minimal surfaces), in dimension 3 by  Kleiner \cite{Kleiner} in 1992 (see also the survey paper by Ritor\'e \cite{Rit} for a variant of Kleiner's arguments), and in dimension 4 by Croke \cite{Croke} in 1984. An interesting  feature of this problem is that  the above proofs have nothing to do one with the other and that they work only for one specific dimension; probably also for this reason such a problem is still open in the general case.} \hfill$\Box$
\end{remark}

 
 \subsection{The case of  minimal submanifolds} \label{Sec:Min} 
 Given   a smoothly immersed submanifold $M^n\hookrightarrow \R^{n+k}$, by the first variation formula for the area functional  we know that for every $\Omega\subset M^n$ open bounded subset with smooth boundary and every smooth vector field $X$ along $\Omega$ it holds
 \be\label{eq:TangDivThm}
 \int_{\Omega} \Div_M X \, d \Ha^{n} = - \int_{\Omega} H \cdot X \, d\Ha^n- \int_{\partial \Omega} X \cdot \nu \, d \Ha^{n-1},
 \ee
 where $H$ is the mean curvature vector of $M$ and  $\nu$ is the inward pointing  conormal to $\Omega$ (i.e. $\nu$ is the unit  vector tangent to $M$, normal to $\partial \Omega$ and pointing inside $\Omega$).
 \\
 
We are interested in the case  $M^n\hookrightarrow \R^{n+k}$ is a minimal submanifold, i.e. $H\equiv 0$, and $\Omega \subset M^n$ is a bounded open subset with smooth boundary $\partial \Omega$.  Let $x_0\in \R^{n+k}$ be  such that 
$$\max_{x \in \bar{\Omega}} |x-x_0|_{\R^{n+k}}= \rad_{\R^{n+k}}(\Omega),$$
and observe that, called $X(x):=x-x_0$, one has  $\Div_M X\equiv n$. By applying \eqref{eq:TangDivThm}, we then infer

\be\label{eq:IsoPDMinExt1}
n \Ha^n(\Omega) = \int_{\Omega} \Div_M X \, d\Ha^n = - \int_{\partial \Omega} X \cdot \nu \, d\Ha^{n-1} \leq \rad_{\R^{n+k}}(\Omega)\Ha^{n-1}(\partial \Omega).
\ee
Notice that equality is achieved if and only if $\Omega$ is the intersection of $M$ with a round ball in $\R^{n+k}$ centered at $x_0$ and $\nu(x)$ is  parallel to $x-x_0$, or in other words if and only if $\Omega$ is a free boundary minimal $n$-submanifold in a ball of $\R^{n+k}$. So we have just proved the following result.

\begin{proposition}\label{thm:IsoPDMinExt}
Let $M^n\hookrightarrow \R^{n+k}$ be  a minimal submanifold  and $\Omega \subset M^n$  a bounded open subset with smooth boundary $\partial \Omega$. Then
$$n \Ha^n(\Omega) \leq  \rad_{\R^{n+k}}(\Omega) \, \Ha^{n-1}(\partial \Omega)$$
with equality if and only if $\Omega$ is a free boundary minimal $n$-submanifold in a ball of $\R^{n+k}$.
\end{proposition}

\begin{remark} [Euclidean isoperimetric inequality on minimal submanifolds] \label{rem:OPMin}
\rm{The statement corresponding to Proposition \ref{thm:IsoPDMinExt} for the isoperimetric problem is the following celebrated conjecture: 
\emph{Let $M^n\subset \R^{m}$ be a minimal $n$-dimensional submanifold and let $\Omega\subset M^n$ be a smooth open subset. Then $\Omega$ satisfies  the Euclidean isoperimetric inequality \eqref{eq:IsopIne}, and equality holds if and only if $\Omega$ is a ball in an affine $n$-plane of $\R^{m}$.} 

To our knowledge the only two solved cases are $i)$ when $\partial \Omega$ lies on an $(m-1)$-dimensional Euclidean sphere centered at a point of $\Omega$ (the argument is by monotonicity, see for instance \cite[Section 8.1]{Choe}) and $ii)$ when $\Omega$ is area minimizing with respect to its boundary $\partial \Omega$ by Almgren \cite{Al86}. Let us mention that a  complete solution of the above conjecture is still not available even for minimal surfaces in $\R^m$, i.e. for $n=2$; however, in the latter situation, the statement is known to be true  in many cases (let us just mention that in case $\Omega$ is a topological disk the problem was solved by Carleman \cite{Car} in 1921, and the case $m=3$ and  $\partial \Omega$ has two connected components was settled much later by Li-Schoen-Yau \cite{LSY}; for more results in this direction and for  a comprehensive overview see the  beautiful survey paper \cite{Choe} by Choe). Let us finally observe that, when $n=2$ and $m=3$, th!
 e above c
 onjecture is a special case of the Aubin Conjecture  recalled in Remark \ref{rem:OPCH}, since  of course the induced metric on a  immersed minimal surface in $\R^3$ has non-positive Gauss curvature; this case was settled in the pioneering work by Weil \cite{Weil}.} \hfill$\Box$
\end{remark}

\begin{remark}[Free boundary minimal submanifolds and critical metrics] \label{rem:FreeBCritMet}
\rm{After a classical work of Nitsche \cite{Ni} in the 80'ies, the last years have seen an increasing interest on free boundary submanifolds  also thanks to recent works of Fraser and Schoen \cite{FS1,FS2} on the topic. 
By definition, a \emph{free boundary  submanifold} $M^n$ of the unit ball $B^{n+k}$, is a proper submanifold which is critical for the area functional with respect to variations of $M^n$ that are allowed to move also  the boundary $\partial M^n$,but  under the constraint $\partial M^n \subset \partial B^{n+k}$.   As a consequence  of the $1^{st}$ variational formula,  such definition forces on one hand the mean curvature to vanish on $M^n\cap B^{n+k}$ and on the other hand   the submanifold to the meet the ambient boundary $\partial B^{n+k}$ orthogonally. These are characterized by the condition that the coordinate functions are Steklov eigenfunctions with eigenvalue 1 \cite[Lemma 2.2]{FS1}; that is, 
$$\Delta x_i = 0 \text{ on } M \text{ and } \nabla_{\nu } x_i = - x_i \text{ on } \partial M.$$
 It turns out that surfaces of this type arise naturally as extremal metrics for the Steklov eigenvalues (see \cite{FS2} for more details); Steklov eigenvalues are eigenvalues of the Dirichlet-to-Neumann map, which sends a given smooth function on the boundary  to the normal derivative of its harmonic extension to the interior.} \hfill$\Box$
\end{remark}

\begin{remark}[Examples of free boundary minimal submanifolds]  \label{rem:ExFreeBounday}
\rm{Let us recall here some well known examples of free boundary minimal submanifolds in the unit ball   $B^{n+k}\subset \R^{n+k}$, for a deeper discussion on the examples below see \cite{FS2}.
\begin{itemize}
\item \emph{Equatorial Disk}.  Equatorial $n$-disks $D^n \subset B^{n+k}$ are the simplest examples of free boundary minimal submanifolds.  By a result of Nitsche \cite{Ni} any simply connected free boundary minimal surface in $B^3$ must be a flat equatorial disk. However, if we admit minimal surfaces of a different topological type, there are other examples, as the critical catenoid described below.
\item \emph{Critical Catenoid}.   Consider the catenoid parametrized on $\R \times S^1$  by the function $$\varphi(t,\theta) = (\cosh t \cos \theta, \cosh t \sin \theta, t)\, .$$
For a unique choice of  $T_0>0$, the restriction of $\varphi$ to $[-T_0, T_0] \times S^1$ defines a minimal embedding into a ball meeting the boundary of the ball orthogonally.  By rescaling  the radius of the ball to 1 we get the critical catenoid in $B^3$. Explicitly, $T_0$ is the unique positive solution of $t = \coth t$.
\item  \emph{Critical M\"obius band}. We think of the M\"obius band $M^2$ as $\R \times S^1$ with the identification $(t, \theta) \sim (-t, \theta+\pi)$. There is a minimal embedding of $M^2$ into $\R^4$ given by
$$\varphi(t,\theta) = (2\sinh t \cos \theta, 2 \sinh t \sin \theta, \cosh 2t \cos 2\theta, \cosh2t \sin 2\theta)\,.$$
For a unique choice of $T_0>0$, the restriction of $\varphi$ to $[-T_0, T_0] \times S^1$ defines a minimal embedding into a ball meeting the boundary of the ball orthogonally. By rescaling  the radius of the ball to 1 we get the critical M\"obis band in $B^4$. Explicitly $T_0$ is the unique positive solution of $\coth t = 2 \tanh 2t$.
\item A consequence of the results of \cite{FS2}  is that for every $k \geq 1$ there exists an embedded free boundary minimal surface in $B^3$ of genus 0 with $k$ boundary components.
\end{itemize}} \hfill$\Box$
\end{remark}



\noindent
Since of course $\rad_{\R^{n+k}}(\Omega) \leq \rad_{M}(\Omega)$, where $\rad_{M}(\cdot)$ is the extrinsic radius in the metric space $(M,\sfd_g)$, we  have a fortiori that 
\be\label{eq:IsoPDMinExt2}
n \Ha^n(\Omega) \leq \rad_{M}(\Omega) \, \Ha^{n-1}(\partial \Omega).
\ee
But in this case the rigidity statement is much stronger, indeed  in case of equality the center of the ball $x_0$ must be a point of  $M$, moreover for every $x \in \partial \Omega$ the segment $\overline{x,x_0}$ must be contained in $M$, therefore  $M$ contains a portion of a minimal cone $\mathcal{C}$ centered at $x_0$. But since by assumption $M$ is  a smooth submanifold and  since the only cone smooth at its origin is an affine subspace, it must be that $M$ contains  a portion of  an affine subspace. By the classical weak unique continuation property for solutions to the minimal submanifold  system, we conclude that $M$ is an affine subspace of $\R^{n+k}$. Therefore we have just proven the next result. 

\begin{proposition}\label{thm:IsoPDMinint}
Let $M^n\hookrightarrow \R^{n+k}$ be  a connected smooth minimal submanifold and $\Omega \subset M^n$  a bounded open subset with smooth boundary $\partial \Omega$. Then
\be\label{eq:IsoPDMinint0}
n \, \Ha^n(\Omega) \leq \rad_{M}(\Omega) \, \Ha^{n-1}(\partial \Omega)
\ee
with equality if and only if $M$ is an affine subspace and $\Omega$ is the intersection of $M$ with a round ball in $\R^{n+k}$ centered at a point of $M$.
\end{proposition}

\begin{remark}
\rm{If we allow $M$ to have conical singularities, then \eqref{eq:IsoPDMinint0} still holds with equality if and only if $M$ is a minimal cone and $\Omega$ is the intersection of $M$ with a round ball in $\R^{n+k}$ centered at a point of $M$. 

Concerning this, recall that  in case $n=2$ and $k=1$ every minimal cone smooth away from the  vertex  is totally geodesic, indeed one of the principal curvatures is always null for cones and so the mean curvature vanishes if and only if all the second fundamental form is null.  Therefore  equality in \eqref{eq:IsoPDMinint0}  is attained if and only if $M^2$ is  an affine plane and $\Omega$ is a flat 2-disk. The analogous result for $n=3$ and $k=1$ is due to Almgren  \cite{Alm} (see also the work of Calabi \cite{Cal}). 

For the general case of  higher dimensions and co-dimensions note that a minimal submanifold $\Sigma^k$ in $S^n$ is naturally the boundary of a minimal submanifold of the ball, the cone $C(\Sigma)$ over $\Sigma$. Using this correspondence it is possible to construct many  non trivial minimal cones:  Hsiang \cite{Hsiang1}-\cite{Hsiang2} gave infinitely many co-dimension 1 examples for $n\geq 4$, the higher co-dimensional problem was investigated in the celebrated paper of Simons \cite{Sim} and the related work of Bombieri-De Giorgi-Giusti \cite{BDG}.} \hfill$\Box$ 
\end{remark}

 arp e non
  posse concludere). If we assume $V$ to be an integer rectifiable varifold   after finitely many steps we have exhausted the  varifold.


\section{The isoperimetric-isodiametric inequality  in manifolds with non-negative Ricci curvature}\label{Sec:Comparison}
In this section we show a comparison result for manifolds with non-negative Ricci curvature which will be used in Section \ref{sec:Existence} to get existence of isoperimetric-isodiametric regions in manifolds  which are  asymptotically locally Euclidean and have  non-negative Ricci (the so called ALE spaces).

\begin{theorem}\label{thm:Ric>0}
Let $(M^n,g)$ be a complete (possibly non compact) Riemannian $n$-manifold with non-negative Ricci curvature.  Let  $B_r\subset M$ be a metric ball  of volume $V=\V(B_r)$, and denote with $B^{\R^n}(V)$ the round ball in $\R^n$ having volume $V$. Then
\begin{equation}
\rad(B_r) \, \P(B_r)=r \, \P(B_r) \leq nV=\rad_{\R^n}(B^{\R^n}(V)) \, \P_{\R^n} (B^{\R^n}(V)).
\end{equation}
Moreover equality holds if and only if $B_r$ is \emph{isometric} to a round ball in the Euclidean space $\R^n$. In particular, for every $V\in (0,\V(M))$ it holds
\begin{equation}\label{eq:CompRic>0}
\inf \{ \rad(\Omega) \P(\Omega) \,:\,  \Omega \subset M, \, \V(\Omega)=V \} \leq n V =  \inf \{ \rad(\Omega) \P(\Omega) \,:\,  \Omega \subset \R^n, \, \V_{\R^n}(\Omega)=V \}, 
\end{equation}
with equality for some $V\in (0,\V(M))$ if and only if every metric ball in $M$ of volume $V$ is isometric to a round ball in $\R^n$. In particular if equality occurs for some $V\in (0,\V(M))$ then  $(M,g)$ is flat, i.e. it has identically zero sectional curvature.
\end{theorem}

\begin{proof}
Let us fix an arbitrary $x_0\in M$ and let $B_r=B_r(x_0)$ be the metric ball in $M$ centered at $x_0$ of radius $r>0$. It is well known that the distance function $\sfd_{x_0}(\cdot):=\sfd(x_0,\cdot)$ is smooth outside the cut locus ${\mathcal C}_{x_0}$ of $x_0$ and that $\mu_g({\mathcal C}_{x_0})=0$. From the co-area formula it follows that for ${\mathcal L}^1$-a.e. $r\geq 0$ one has $\cH^{n-1} ({\mathcal C}_{x_0} \cap \de B_r(x_0))=0$ and, since the cut locus is closed by definition, we get that for  ${\mathcal L}^1$-a.e. $r\geq 0$
the distance function  $\sfd_{x_0}(\cdot)$ is smooth on an open subset of full ${\cH}^{n-1}$-measure on  $\de B_r(x_0)$.  
\\Let us first assume that $r>0$ is one of these regular radii, the general case will be settled in the end by an approximation argument. It is immediate to see that on $\de B_r(x_0)\setminus {\mathcal C}_{x_0}$ we have  $|\nabla \sfd_{x_0}|=1$ and thus  $\de B_r(x_0)\setminus {\mathcal C}_{x_0}$ is a smooth hypersurface. In particular, since  $\cH^{n-1}(\de B_r(x_0)\cap {\mathcal C}_{x_0})=0$, we have that $B_r(x_0)$ is a  finite perimeter set  whose reduced boundary is contained $\de B_r(x_0)\setminus {\mathcal C}_{x_0}$.  Called $\nu$ the inward pointing unit  normal to $\de B_r(x_0)$ on the regular part $\de B_r(x_0)\setminus {\mathcal C}_{x_0}$, from the Gauss Lemma we have
\begin{equation}\label{eq:GaussLemma}
\nu=- \nabla \sfd_{x_0}, \quad \text{on }  \de B_r(x_0)\setminus {\mathcal C}_{x_0}.
\end{equation}
Therefore, called  $u:=\frac{1}{2} \sfd^2_{x_0}$, we infer
\begin{eqnarray}
r \P(B_r(x_0))&=& - \int_{\de B_r(x_0)\setminus {\mathcal C}_{x_0}}  \sfd_{x_0}(x) \, g(\nabla \sfd_{x_0}(x), \nu(x)) \, d \cH^{n-1}(x) =  - \int_{\de B_r(x_0) \setminus {\mathcal C}_{x_0}}   g(\nabla u, \nu) \, d \cH^{n-1} \nonumber \\
&=&- \lim_{\varepsilon \downarrow 0} \int_{\de B_r(x_0) \setminus {\mathcal C}_{x_0}}    g(\nabla u_\varepsilon, \nu) \, d \cH^{n-1} \nonumber 
\end{eqnarray}
where $u_\varepsilon \in C^2(M)$ is a approximation by convolution of $u$ such that $\| \nabla u_\varepsilon -\nabla u\|_{L^\infty(\de B_r(x_0), \cH^{n-1})}\to 0$,  $\Delta u_\varepsilon  \to \Delta u $ in $C^0_{loc}(M\setminus {\mathcal C}_{x_0})$ and $\Delta u_\varepsilon \leq n$ where in the last estimate we used the global Laplacian comparison stating that  $\Delta u$ is a Radon measure with  $\Delta u\leq n \,  \mu_g$.
More precisely, one has that $\Delta u \llcorner M\setminus {\mathcal C}_{x_0}$ is given by $\mu_g$ multiplied by a smooth function bounded above by $n$, and the singular part $(\Delta u)^s$ of $\Delta u$ is a non-positive measure  concentrated on ${\mathcal C}_{x_0}$.    Now  $\nabla u_\varepsilon$ is  a $C^1$ vector field and we can apply the Gauss-Green formula for finite perimeter sets \cite[Theorem 3.36]{AFP} to infer
 \begin{eqnarray}
r \P(B_r(x_0)) &=&  \lim_{\varepsilon \downarrow 0} \int_{{B_r}(x_0)}   \Delta u_\varepsilon  \, d\mu_g  =  \lim_{\varepsilon \downarrow 0} \int_{{B_r}(x_0)\setminus {\mathcal C}_{x_0}}   \Delta u_\varepsilon  \, d\mu_g \leq    \int_{{B_r}(x_0)\setminus {\mathcal C}_{x_0}}   \limsup_{\varepsilon \downarrow 0}  \Delta u_\varepsilon  \, d\mu_g   \nonumber \\
&=&  \int_{{B_r}(x_0)\setminus {\mathcal C}_{x_0}} \Delta u \, d \mu_g   \leq n \, \V(B_r), \label{eq:PfRic>0}
\end{eqnarray}
 where the first  inequality we used Fatou's Lemma combined with the upper bound   $\Delta u_{\varepsilon}\leq n$ and the last inequality is ensured by the local Laplacian Comparison Theorem. 
 Notice that if equality occurs then $\Delta u= n \, \mu_g$ on $B_r(x_0) \setminus {\mathcal C}_{x_0}$ and, by analyzing the equality in Riccati equations,  it is well known that this implies $B_r(x_0)$ to be isometric to the round ball in $\R^n$.

If now $r>0$ is a singular radius, in the sense that $\cH^{n-1}(\de B_r(x_0)\cap {\mathcal C}_{x_0}) >0$, then by the above discussion we can find a sequence of regular radii $r_n \to r$ and, by the lower semicontinuity of the perimeter under $L^1_{loc}$ convergence \cite[Proposition~3.38]{AFP} combined with \eqref{eq:PfRic>0} which is valid for $B_{r_n}(x_0)$, we infer
 \begin{eqnarray}
r \P(B_r(x_0)) &\leq& \liminf_{n \to \infty} r_n  \P(B_{r_n}(x_0)) \leq \liminf_{n \to \infty}  \int_{B_{r_n}(x_0)\setminus {\mathcal C}_{x_0}}  \Delta u \, d \mu_g \leq   \limsup_{n \to \infty}  \int_{M \setminus {\mathcal C}_{x_0}}  \chi_{B_{r_n}(x_0)} \, \Delta u  \, d \mu_g  \nonumber \\
&\leq&  \int_{M \setminus {\mathcal C}_{x_0}}  \limsup_{n \to \infty}   \chi_{B_{r_n}(x_0)} \, \Delta u  \, d \mu_g =   \int_{{B_r}(x_0)\setminus {\mathcal C}_{x_0}} \Delta u \, d \mu_g   \leq n \, \V(B_r),   \label{eq:PfRic>0gen}
\end{eqnarray}
where in the first inequality of the second line we used Fatou's Lemma (we are allowed since $\chi_{B_{r_n}(x_0)} \, \Delta u  \leq n$ on $M \setminus {\mathcal C}_{x_0}$), and the last inequality follows again by local Laplacian comparison.  Notice that, as before, equality in \eqref{eq:PfRic>0gen}  forces  $\Delta u= n \, \mu_g$ on $B_r(x_0) \setminus {\mathcal C}_{x_0}$ and then $B_r(x_0)$ is isometric to a Euclidean ball.

The second part of the statement clearly follows from the first part combined with the Euclidean isoperimetric-isodiametric inequality \eqref{eq:IsoPDRn1}.
\end{proof}

\section{Existence of isoperimetric-isodiametric regions} \label{sec:Existence}
In Section \ref{Sec:CH} we have seen explicit isoperimetric-inequalities in some special situations: Cartan-Hadamard spaces and minimal submanifolds. In the present section we investigate the existence of optimal shapes: as it happens also for the isoperimetric problem, we will find that if the ambient manifold is compact an optimal set always exists but if the ambient space is non-compact the situation changes dramatically. The subsequent sections will be devoted to establish the sharp regularity for the  optimal sets.

\subsection{Notation}\label{Sec:notation}
Let $(M^n,g)$ be a  complete Riemannian manifold and denote by $d_g$ the geodesic distance, by $\mu_g$ the 
measure associated to the Riemannian volume form and by ${\mathfrak X}(M)$ the smooth vector fields.
Given a measurable subset $E \subset M$, the perimeter of $E$ is denoted by $\P(E)$ and is given by the following 
formula
\[
\P(E) := \sup \left\{ \int_{E} \div X\,\d\mu_g\,:\, X \in {\mathfrak X}(M),\; \supp(X) \subset\subset 
M,\;\|X\|_{L^{\infty}(M,g)}\leq 1 \right\},
\]
and, for any open subset $\Omega \subset M$, we write $\P(E, \Omega)$ when the fields $X$ are restricted to have 
compact 
support in $\Omega$.   It is out of the scope of this paper to discuss the theory of finite perimeter sets; standard references are \cite{AFP}, \cite{EG} and \cite{Maggi}.  

Since from now on we will work with sets of finite perimeter, which are well defined up to subsets of measure zero,  we will adopt the following definition of extrinsic radius of a measurable subset $E\subset M$:
\[
\rad(E) := \inf\left\{r>0 \;:\; \mu_g(E\setminus B_r(z_0)) = 0 \;\text{for some } z_0 \in M\right\},
\]
where $B_r(z_0)$ denotes the open metric ball with center $z_0$ and radius $r>0$. A metric ball $B_r(z_0)$ satisfying  $\mu_g(E\setminus B_r(z_0)) = 0$, is called an \emph{enclosing ball} for $E$.

We consider the following minimization problem: for every fixed $V \in (0, \mu_g(M))$, 
\begin{equation}\label{e:min problem}
\min \; \Big\{ \rad(E)\,\P(E) \;:\; E \subset M,\;\mu_g(E) = V\Big\},
\end{equation}
and call the minimizers of \eqref{e:min problem} isoperimetric-isodiametric sets (or regions).

\subsection{Existence of  isoperimetric-isodiametric regions in compact manifolds}
Let us start with the following lemma,  stating the lower semi continuity of the extrinsic radius under $L^1_{loc}$ convergence. 

\begin{lemma}[Lower semi-continuity of extrinsic radius under $L^1_{loc}$ convergence]\label{lem:lscRad}
Let $(M,g)$ be a  (non necessarily compact) Riemannian manifold and let   $(E_k)_{k \in \N\cup\{\infty\}}$ be a sequence of measurable subsets such  that $\chi_{E_k} \to \chi_{E_\infty}$ in $L^1_{loc}(M, \mu_g)$. Then
$$\rad(E_\infty)\leq \liminf_{k \in \N} \rad(E_k). $$
\end{lemma}

\begin{proof}
Without loss of generality we can assume  $\liminf_{k \in \N} \rad(E_k)<\infty$ so, up to selecting a subsequence, we can assume $\chi_{E_k}\to  \chi_{E_\infty}$ a.e. and  $\lim_{k \uparrow +\infty} \rad(E_k) = \ell <\infty$.
Let $B_k := B_{\rad(E_k)}(x_k)$ be enclosing balls for $E_k$. Then two cases can occur. Either $x_k$ is unbounded, i.e.~$\sup_k d_g(x_k,\bar x) = \infty$ for any $\bar x \in M$, in which case it follows that $E_\infty = \emptyset$ and the conclusion of the lemma is proved. Or there exists $x_\infty \in M$ such that, up to passing to a subsequence, $x_k \to x_\infty$. In this case it is readily verified that
\[
\mu_g\big(E_k \setminus B_{\rad(E_k)+|x_k-x_\infty|}(x_\infty)\big) = 0
\]
from which it follows, by taking the limit as $k\to +\infty$, that $\mu_g\big(E_\infty\setminus B_\ell(x_\infty)\big)=0$, which by definition implies that $\rad(E_\infty) \leq \ell$.
\end{proof}

\noindent
The next theorem is a general existence result for minimizers of the problem  \eqref{e:min problem}, as special cases it will be applied in Corollary \ref{cor:ExCpt} to compact  manifolds and in Theorem \ref{thm:ExistRic>0} for asymptotically locally Euclidean manifolds (ALE for short) having non-negative Ricci curvature. Let us observe that the existence of a minimizer in a non-compact manifold for the classical isoperimetric problem is much harder due to the possibility of  ``small tentacles'' going to infinity in a minimizing sequence; this difficulty is simply not there in the  isoperimetric-isoperimetric problem we are considering, since it would imply the radius to go to infinity.  We believe that this simplification, together with sharp inequalities obtained in the previous section, is another motivation to look at   the isoperimetric-isoperimetric inequality  since it appears more manageable in many situations  than the classical isoperimetric one.

\begin{theorem}[Sufficient conditions for existence of isoperimetric-isodiametric regions] \label{thm:Exist}
Let $(M^n,g)$ be a possibly non compact Riemannian $n$-manifold satisfying the following two  conditions:
\begin{enumerate}
\item $\liminf_{r\to 0^+} \sup_{x \in M} \mu_g(B_r(x))=0$.
\item  There exists $\varepsilon_0>0$ and a  function 
$$\Phi_{Isop}:[0,\varepsilon_0)\to \R^+ \; \text{ with } \; { \lim}_{t\downarrow 0}  \Phi_{Isop}(t)=0, $$
 such that for every finite perimeter set $E\subset M$ with $\P(E)< \varepsilon_0$ the  weak isoperimetric  inequality $\mu_g(E)\leq \Phi_{Isop} (\P(E))$ holds.
\end{enumerate}
Let $V\in (0,\mu_g(M))$ be fixed and let $(E_k)_{k \in \N}\subset M$ be a sequence of finite perimeter sets satisfying 
\be\label{eq:HpEkComp}
\mu_g(E_k)=V, \, \forall k \in \N, \; \text{ and } \; \sup_{k\in \N} \Big( \rad(E_k) \, \P(E_k)\Big) <\infty. 
\ee
Then there exist $R>0$ and a sequence $(x_k)_{k \in N}$ of points in $M$ such that $\mu_g(E_k\setminus B_R(x_k))=0$, i.e. $B_R(x_k)$ are inclosing balls for $E_k$.

In particular, if there exists  a minimizing sequence $(E_k)_{k \in \N}$ for the problem \eqref{e:min problem} relative to some fixed $V\in (0,\mu_g(M))$ such that $\mu_g(E_k\cap K)>0$ for infinitely many $k$ and a fixed compact subset $K\subset M$,  then  there exists an isoperimetric-isodiametric region of volume $V$.
 \end{theorem}

\begin{proof}
We start the  proof by the following two claims.

\emph{Claim 1:} $\inf_{k} \rad(E_k)>0$.
\\Otherwise, up subsequences in $k$,  there exist  $r_k\downarrow 0$ and $x_k\in M$ such that  $\mu_g(E_k\setminus B_{r_k}(x_k))=0$. But then the assumption (1) implies $\mu_g(E_k)\leq \mu_g(B_{r_k}(x_k))=0$, contradicting \eqref{eq:HpEkComp}.

\emph{Claim 2:} $\inf_{k} \P(E_k)>0$.
\\Otherwise, by the  assumption (2) we get  $\mu_g(E_k) \leq  \Phi_{Isop} (\P(E_k))\to 0$, contradicting again \eqref{eq:HpEkComp}.

Combining the two claims with  \eqref{eq:HpEkComp}, we infer that there exists $C>1$ such that
\be\label{eq:compEk}
\frac{1}{C} \leq \P(E_k) \leq C \quad \text{and} \quad \frac{1}{C} \leq \rad(E_k) \leq C,
\ee
so that  the first part of the proposition is proved.
\\If now there exists a compact subset $K\subset M$ such that $\mu_g(E_k\cap K)>0$ for infinitely many $k$ then by  \eqref{eq:compEk}, up to enlarging $K$ and selecting a subsequence in $k$, we can assume $\mu_g(E_k \setminus K)=0$. 
But then the characteristic functions $(\chi_{E_k})_{k \in \N}$ are  pre-compact in $L^1(K,\mu_g)$ 
since the total variations of $\chi_{E_k}$ are equi-bounded  by \eqref{eq:compEk} (cf.~\cite[Theorem~3.23]{AFP}). The thesis then follows by the lower semicontinuity of the perimeter under $L^1_{loc}$ convergence (cf.~\cite[Proposition~3.38]{AFP}) combined with Lemma \ref{lem:lscRad}.
\end{proof}

Clearly if the manifold is compact all the assumptions of  Theorem    \ref{thm:Exist} are satisfied and we can state the following corollary.
\begin{corollary}[Existence of  isoperimetric-isodiametric regions in compact  manifolds]\label{cor:ExCpt} 
Let $(M^n,g)$ be a compact Riemannian manifold. Then for every $V\in (0, \mu_g(M))$ there exists a minimizer of the problem   \eqref{e:min problem},  in other words there exists an isoperimetric-isodiametric region of volume $V$.
\end{corollary}



\subsection{Existence of  isoperimetric-isodiametric regions in non-compact  ALE spaces with non-negative Ricci curvature}
Let us start by recalling the notion of pointed $C^0$-convergence of metrics.
\begin{definition} \label{def:C0Conv}
Let $(M^n,g)$ be a  smooth complete Riemannian manifold and fix $\bar{x} \in M$. A sequence of pointed smooth complete Riemannian $n$-manifolds  $(M_k, g_k, x_k)$ is said to converge in
the pointed $C^{0}$-topology to the manifold $(M,g,\bar{x})$, and we  write $(M_k, g_k, x_k)\rightarrow (M,g,\bar{x})$, if for every $R > 0$ we can find
a domain $\Omega_R$ with $B_R(\bar{x})\subseteq\Omega_R\subseteq M$, a natural number $N_R\in\mathbb{N}$, and $C^{1}$-embeddings $F_{k,R}:\Omega_R\rightarrow M_k$ for large $k\geq N_R$ such that
$B_R(x_k)\subseteq F_{k,R} (\Omega_R)$ and $F_{k,R}^*(g_k)\rightarrow g$ on $\Omega_R$ in the $C^{0}$-topology.
\end{definition}

\begin{theorem}\label{thm:ExistRic>0}
Let $(M,g)$ be a complete Riemannian $n$-manifold with non-negative Ricci curvature and  fix any reference point  $\bar{x}\in M$. Assume that for any diverging sequence of points  $(x_k)_{k \in N} \subset M$, i.e. $\sfd(x_k, \bar{x})\to \infty$, the sequence of pointed manifolds $(M, g, x_k)$ converges in the pointed $C^{0}$-topology to the Euclidean space $(\R^n, g_{\R^n}, 0)$. 
\\Then for every $V \in [0, \mu_g(M))$ there exists a minimizer of the problem   \eqref{e:min problem},  in other words there exists an isoperimetric-isodiametric region of volume $V$.
\end{theorem}

\begin{proof}
Since volume and perimeter involve  only the metric tensor $g$ and not its derivatives, the hypothesis  on  the manifold $(M,g)$ of being  $C^0$-locally asymptotic to $\R^n$ implies directly that assumptions  (1) and (2) of Theorem \ref{thm:Exist} are satisfied.
Therefore  the thesis will be a consequence of  Theorem \ref{thm:Exist} once we show the following:  given $E_k\subset M$ a minimizing  sequence of the problem   \eqref{e:min problem} for some fixed volume $V\in [0, \mu_g(M))$, then  there exists a compact subset $K \subset M$ such that $\mu_g(E_k\cap K)>0$ for infinitely many $k$. We will show that if this last statement is violated then $(M,g)$ is flat and minimizers are metric balls of volume $V$. 

By  the first part of  Theorem \ref{thm:Exist} we know that  there exist $R > 0$ and a sequence $(x_k)_{k\in \N}$ of points in $M$ such that $\mu_g(E_k \setminus B_R(x_k))=0$, i.e. $B_R(x_k)$ are inclosing balls for $E_k$.
\\Fixed any reference point $\bar{x}\in M$, if $\liminf_k \sfd(x_k, \bar{x})$ then clearly we can find a compact subset $K \subset M$ such that $\mu_g(E_k\cap K)>0$ for infinitely many $k$ and the conclusion follows from the last part of Theorem  \ref{thm:Exist}.
 So assume  that $\sfd(\bar{x}, x_k)\to \infty$. Since $M$ is $C^0$-locally asymptotic to $\R^n$, combining Definition \ref{def:C0Conv} with  the Euclidean isoperimetric-isodiametric inequality \eqref{eq:IsoPDRn1}, we get  that
 \begin{equation}\label{eq:contr}
 \liminf_{k \to \infty} \rad(E_k) \, \P(E_k) \geq n \,V. 
 \end{equation}
 But since $(M,g)$ has non-negative Ricci curvature, the comparison estimate \eqref{eq:CompRic>0} yields that 
 \be\label{eq:radP>nV}
 \lim_{k\to \infty}   \rad(E_k) \, \P(E_k)= \inf \{ \rad(\Omega) \P(\Omega) \,:\,  \Omega \subset M, \, \V(\Omega)=V \} \leq n V.
 \ee
 The combination of  \eqref{eq:contr} with  \eqref{eq:radP>nV} clearly implies 
 $$\inf \{ \rad(\Omega) \P(\Omega) \,:\,  \Omega \subset M, \, \V(\Omega)=V \} = n V .$$
 The rigidity statement of Theorem {\ref{thm:Ric>0}} then gives that any metric  ball in $(M,g)$ of volume $V$ is isometric to a   round ball  in $\R^n$, and therefore in particular  is a minimizer of the problem  \eqref{e:min problem}.
\end{proof}

\subsection{Examples of non-compact spaces where existence of   isoperimetric-isodiametric regions fails}

\begin{example}[Mimimal surfaces with planar ends]\label{ex:NoExMinimal}
\rm{
If $M\subset \R^3$ is an helicoid, or more generally a minimal surface with planar ends, then it is in particular $C^0$-locally asymptotic to $\R^2$ in the sense of Definition  \ref{def:C0Conv}. Then, if we consider a sequence of metric balls  $B_{r_k}(x_k) \subset M$ of fixed volume $V>0$ such that $x_k\to \infty$ we get $\lim_{k\to \infty} \rad(B_{r_k}(x_k)) \V(B_{r_k}(x_k)) = 2\, V$. In particular, for every $V>0$ we have
$$\inf \{ \rad(\Omega) \P(\Omega) \,:\,  \Omega \subset M, \, \V(\Omega)=V \} \leq 2 V. $$
But then Proposition  \ref{thm:IsoPDMinint} implies that the infimum is never achieved, or more precisely it is achieved if and only if $M$ is an affine subspace.
\\The same argument holds for any minimal $n$-dimensional sub-manifold in $\R^m$ with ends which are $C^0$-locally asymptotic to $\R^n$.} \hfill$\Box$
\end{example}

\begin{example} [ALE spaces of negative sectional curvature]\label{ex:NoExALE}
\rm{
Let $(M^n,g)$ be a simply connected non-compact Riemannian manifold with negative sectional curvature and assume that $(M,g)$ is  $C^0$-locally asymptotic to $\R^n$ in the sense of Definition  \ref{def:C0Conv}.  Then, if we consider a sequence of metric balls  $B_{r_k}(x_k) \subset M$ of fixed volume $V>0$ such that $x_k\to \infty$ we get $\lim_{k\to \infty} \rad(B_{r_k}(x_k)) \V(B_{r_k}(x_k)) = n\, V$. In particular, for every $V>0$ we have
$$\inf \{ \rad(\Omega) \P(\Omega) \,:\,  \Omega \subset M, \, \V(\Omega)=V \} \leq n V. $$
But then Proposition  \ref{thm:CH} implies that the infimum is never achieved, or more precisely it is achieved  by a region $\Omega$ if and only if $\Omega$ is isometric to a Euclidean region, which is forbidden since $M$ has negative sectional curvature.
} \hfill$\Box$
\end{example}

\section{Optimal regularity of isoperimetric-isodiametric regions} \label{sec:OptReg}
In this last section we establish the optimal regularity for the isoperimetric-isodiametric regions, i.e. the minimizers of problem \eqref{e:min problem}, under the assumption that the enclosing ball is regular.

\subsection{$C^{1,\frac12}$ regularity}\label{s:C1alpha}
\subsubsection{First properties}\label{s:preliminaries}
Let $E$ be a minimizer of the isoperimetric--isodiametric problem in $(M,g)$ with volume $\mu_g(E) = V>0$.
Let $x_0 \in M$ satisfy $\mu_g(E\setminus B_{\rad(E)}(x_0)) = 0$ and, for the sake of 
simplicity, we fix the notation $B:=B_{\rad(E)}(x_0)$ for an enclosing ball.
In the sequel, we always assume that $B$ has regular boundary and we assume to be in the non-trivial 
case $\mu_g(B\setminus E) 
>0$.

By the very definition of isoperimetric-isodiametric sets, we have that 
\begin{equation}\label{e:volume minimizing}
\P(E) \leq \P(F) \quad \forall\;F \sdif E \subset\subset B\; : \;\mu_{g}(F) = V.
\end{equation}
In particular, $E$ is a minimizer of the perimeter with constrained volume in $B$,
and therefore we can apply the classical regularity results (see, for example, \cite[Corollary~3.8]{Morgan})
in order to deduce that there exists a relatively closed set $\sing(E) \subset B$ such 
that
$\dim_{\cH}(\sing(E)) \leq n-8$ and $\de E \cap B \setminus \sing(E) $ is a smooth $(n-1)$-dimensional 
hypersurface.

Moreover, by the first variations of the area functional under volume constraint, one deduces that the mean curvature 
is constant on the regular part of the boundary: i.e. there exits $H_0 \in \R$ such that
\begin{equation}\label{e:mean curvature const}
\vec H_E(x) = H_0\,\nu_E \quad \forall\; x \in \de E \cap B \setminus \sing(E),
\end{equation}
where
\[
\vec H_E(x) := \sum_{i=1}^{n-1} \nabla_{\tau_i} \tau_i,
\]
for $\{\tau_1, \ldots, \tau_{n-1} \}$ a local orthonormal frame of $\de E$ around $x\in \de E\cap B \setminus 
\sing(E)$, $\nu_E$ the interior normal to $E$ and $\nabla$ the Riemannian connection on $(M,g)$.

In this section we prove the following.

\begin{proposition}\label{p:regularity}
Let $E\subset M$ be an isoperimetric-isodiametric set and $x_0 \in M$ be such that $\mu_g(E \setminus 
B_{\rad(E)}(x_0)) = 0$. Assume that $B:= B_{\rad(E)}(x_0))$ has smooth boundary.
Then, there exists $\delta>0$ such that $\de E \setminus B_{\rad(E) - \delta}(x_0)$ is $C^{1,\frac12}$ regular.
\end{proposition}

\begin{remark}
\rm{In particular, given the partial regularity in $B$ as explained in \S~\ref{s:preliminaries}, we conclude that $E$ is a 
closed set whose boundary is $C^{1,\frac12}$ regular except at most a closed singular set $\sing(E)$ of dimension less 
or equal to $n-8$. }\hfill$\Box$
\end{remark}

\subsubsection{Almost minimizing property}
The main ingredient of the proof of Proposition~\ref{p:regularity} is the following 
almost-minimizing property.

\begin{lemma}\label{l:almost min}
Let $E$ be an isoperimetric-isodiametric set in $M$ and let $B$ denote an enclosing ball as above.
There exist constants $C, r_0 >0$ such that,
for every $x\in B$ and for every $0<r<r_0$, the following holds
\begin{equation}\label{e:almost min}
\P(E) \leq \P(F) + C\,r^{n} \quad \forall\; F \sdif E \subset \subset B_r(x).
\end{equation}
\end{lemma}

\begin{remark} 
\rm{Note that $B_r(x)$ is not necessarily contained in $B$.}\hfill$\Box$
\end{remark}

\begin{proof}
We start fixing parameters $\eta, c_1>0$ and two points $y_1, y_2 \in B$
such that $d_g(y_1,y_2)>4\,\eta$, $B_{4\eta}(y_1) \subset B$, $B_{4\eta}(y_2) \subset B$ and
\begin{equation}\label{e:density}
\P(E, B_{\eta}(y_i)) > c_1 \quad i=1,2.
\end{equation}
Note that the possibility of such a choice is easily deduced from the regularity
of the previous subsection, or more elementary from the density estimates
for sets of finite perimeter in points of the reduced boundary.
Set for simplicity of notation $D_i := B_{\eta}(y_i)$.
By a result by Giusti \cite[Lemma~2.1]{Giusti}, there exist $v_0, C_1>0$ such
that, for every $v \in \R$ with $|v| < v_0$ and for every $i=1,2$,
there exists $F_i$ which satisfies the following
\begin{equation}\label{e:aggiusta volumi}
\begin{cases}
F_i \sdif E \subset D_i,\\
\mu_g(F_i) = \mu_g(E) + v,\\
\P(F_i) \leq \P(E) + C_1\,v.
\end{cases}
\end{equation}
Note that in \cite[Lemma~2.1]{Giusti}  the property \eqref{e:aggiusta volumi} is proven
in the Euclidean space with the flat metric, but the proof remains unchanged
in a Riemannian manifold (up to a suitable choice of the constants $v_0, C_1$).

Next, let $r_0>0$ be a constant to be fixed momentarily such
that $r_0 <\eta$ and
 \begin{equation} \label{eq:supmuBr}
 \sup_{x \in B} \mu_g(B_r(x)) \leq C_2 \, r^n  < v_0, \quad \forall r \in [0,r_0]
 \end{equation}
 for some $C_2>0$ depending just on $B$ and $r_0$. \fn
Since $d_g(y_1,y_2)>4\,\eta$, for every $x \in B$, $B_{r_0}(x)$ cannot intersect both $D_1$ and $D_2$: 
therefore, without loss of generality, we can assume $B_{r_0}(x) \cap D_1 = \emptyset$.
If $r<r_0$ and $F\subset M$ is any set such that $F\sdif E \subset\subset B_r(x)$,
we consider $F' := F\cap B$. Note that
$F' \subset B$ and moreover
\[
|\mu_g(F') - \mu_g(E)| \leq  \mu_g(B_r(x)) \leq C_2 \, r^n <v_0.
\]
According to \eqref{e:aggiusta volumi} we can then find $F'' \subset B$
such that  $\mu_g(F'')=\mu_g(E)$, \fn  $F''\sdif F' \subset\subset D_1$ and
\begin{equation}\label{e:perim controlled}
\P(F'') \leq \P(F') + C_1 |\mu_g(F') - \mu_g(E)|.
\end{equation}
Using the fact that $E$ minimizes the perimeter among compactly
supported perturbation in $\bar B$, we deduce that
\begin{align}\label{e:qm}
\P(E) &\leq \P(F'') \stackrel{\eqref{e:perim controlled}}{\leq}
\P(F') + C_1 |\mu_g(F') - \mu_g(E)|\notag\\
& \leq \P(F) + \P(B) - \P(F\cup B) +   C_2\,r^n. \fn
\end{align}
Next note that, if $\de B$ is $C^{1,1}$ regular, then one can choose
$r_0>0$ such that the following holds: there exists a constant $C_3>0$
such that, for every $x \in B$ and for every $r\in (0,r_0)$,
\begin{equation}\label{e:smooth qm}
\P(B) \leq \P(G) + C_3\,r^{n} \quad \forall \; G\sdif B \subset\subset B_{r}(x).
\end{equation}
In order to show this claim, it it enough to take $r_0$ small enough (in particular smaller than  half \fn the 
injectivity radius) in such a way that, for every $p \in \de B$, there exists a co-ordinate chart $\phi: B_{ 2 r_0 \fn}(p) \to 
\R^n$ such that $ \phi(\de B) \fn \subset \{x_n =0\}$ and $\phi$ is a $C^{1,1}$ diffeomorphism with  $\d\phi(p) \in SO(n), \phi(p)=0$ and $g(0) = 
\Id$, $g$ being the metric tensor in the coordinates induced by $\phi$ \fn.
Indeed, in this case we have that $\P(B, B_r(p)) \leq (1+Cr)\omega_{n-1}r^{n-1}$ for every $r<r_0$ and, for every $G$ 
such that $G\sdif B \subset\subset B_{r}(p)$,  $$\P(G, B_r(p)) \geq (1-Cr) \P(\textup{proj}(\phi(G)), \phi(B_r(p))) \geq (1-Cr) \omega_{n-1}r^{n-1},$$ \fn
where $\textup{proj}$ denotes the orthogonal Euclidean projection on $\{x_n=0\}$ and we have used the regularity of 
$\phi$.

Applying \eqref{e:smooth qm} to $G = F \cup B$ and using \eqref{e:qm}, we conclude the proof.
\end{proof}

\subsubsection{Proof of Proposition~\ref{p:regularity}}
Now we are in the position to apply a result by Tamanini \cite[Theorem~1]{Tam} (the result is proved
in $\R^n$ with a flat metric, but the proof is unchanged in
a Riemannian manifold) in order to give a proof of the above proposition.

To this aim, we start considering any point $p \in \de B \cap \de E$;  we denote with $Exp_p:T_pM \to M$  the exponential map and we let $r_0>0$ be less then the injectivity radius. \fn Since by Lemma~\ref{l:almost min} the set  $E$ is an 
almost minimizer of the perimeter, the rescaled sets 
\begin{equation} \label{eq:defEpr}
 E_{p,r} : = \frac{Exp_p^{-1}(E\cap B_{r_0}(p))}{r} \subset T_pM \simeq \R^n 
 \end{equation} \fn 
 converge up to passing to a suitable 
subsequence to a minimizing cone  $C_\infty$ in the Euclidean space (see \cite[Theorem~28.6]{Maggi}).   Moreover, since $E$ is enclosed by $B$ and $\de B$ is  $C^{1,1}$,  it is immediate to check that if $r_0>0$ is chosen small enough in \eqref{eq:defEpr}, then $ C_\infty \subset \{x: g(\nu_B(p),x) 
\geq 0\}$,  \fn we deduce that every tangent cone to 
$E$ at $p$ needs to be contained in a half-space, and therefore by the Bernstein theorem is flat 
(cf.~\cite[Theorem~17.4]{Giusti-book}).
This implies that every such point $p$ is a point of the reduced boundary of the set (see \cite[Definition 3.54]{AFP}) 
and therefore we can apply the aforementioned result by Tamanini to conclude that $\de E$ is a 
$C^{1,\frac{1}{2}}$ regular 
hypersurface in $B_r(p)$ for every $p \in \de B \cap \de E$ and for every $r<\frac{r_0}{2}$.
By a simple covering argument, the conclusion of the corollary follows.

\subsection{$L^\infty$ estimates on the mean curvature of the minimizer}\label{s:LinftyEst}
In this section we prove that the boundary of $E$ has generalized mean curvature in the sense of varifolds which is 
bounded in $L^\infty$. To this aim, we compute the first variations of the perimeter of $E$ along suitable 
diffeomorphisms.

\subsubsection{First variations}
We start fixing two points $y_1,y_2 \in \de E \cap B \setminus \sing(E)$ and a real number $\eta>0$ such that 
$B_{4\eta}(y_1) \subset B$, $B_{4\eta}(y_2) \subset B$ and
\[
B_{4\eta}(y_1) \cap B_{4\eta}(y_2) = B_{4\eta}(y_1) \cap \sing(E) = B_{4\eta}(y_2) \cap \sing(E) = \emptyset.
\]
Note that such a choice is possible in the hypothesis that $\mu_g(B\setminus E) >0$ because of the partial regularity in 
\S~\ref{s:preliminaries}.
Let $X \in \mathfrak{X}(M)$ be a vector field with support contained in a metric ball $B_{\eta}(y)$ for some $y \in 
M$.
Clearly, $B_{\eta}(y)$ cannot intersect both $B_{2\eta}(y_1)$ and $B_{2\eta}(y_2)$, because $d_g(y_1,y_2)\geq  8\eta$;  \fn
therefore, without loss of generality let us assume that 
$B_{\eta}(y) \cap B_{2\eta}(y_1) = \emptyset$.
It is not difficult to construct a smooth vector field ${Y}$ supported 
in $B_{\eta}(y_1)$ such that the generated flow $\{\Phi^{{Y}}_t\}$ satisfies the following properties for small $|t|$:
\begin{equation}\label{eq:volumi}
\mu_g(\Phi^{{Y}}_t \circ \Phi^{{X}}_t (E))= \mu_g(E). 
\end{equation}
Note that the generated flows $\{\Phi^{{X}}_t\}_{t\in \R}$ and $\{\Phi^{{Y}}_t\}_{t\in \R}$ are well-defined and for 
$|t|$ sufficiently small are diffeomorphisms of $M$.
Moreover, $\Phi^{{Y}}_t \circ \Phi^{{X}}_t (E) \subset B_{\rad(E) + |t|\|X\|_{\infty}}$.
We can then deduce that
\begin{align}\label{e:nuovo funzionale}
\rad(E) \P(E) &\leq \rad\big( \Phi^{{Y}}_t \circ \Phi^{{X}}_t (E)\big) \P\big(\Phi^{{Y}}_t \circ \Phi^{{X}}_t 
(E)\big)\notag\\
& \leq \big(\rad(E) + |t|\|X\|_{\infty} 
\big)\P\big(\Phi^{{Y}}_t \circ \Phi^{{X}}_t (E)\big) =: f(t).
\end{align}
Taking the derivative of the last functional as $t\downarrow 0^+$ and as $t \uparrow 0^-$, by the well-known 
computation of the first variations of the area we infer that
\begin{align}\label{e:destra}
0 &\leq \lim_{t\downarrow 0^+} \frac{f(t) - f(0)}{t} \notag\\
&= \|X\|_{\infty} \P(E) + \rad(E)\,\int_{\de E} \div_{\de 
E} X\, \d\cH^{n-1} - \int_{\de E} g\big(\vec{H}_E, Y \big)\,\d\cH^{n-1}
\end{align}
\begin{align}\label{e:sinistra}
0&\geq \lim_{t\uparrow 0^-}\frac{f(t) - f(0)}{t} \notag\\
&= - \|X\|_{\infty} \P(E) + \rad(E)\,\int_{\de E} \div_{\de E} 
X\,\d\cH^{n-1}- \int_{\de E} g\big(\vec{H}_E, Y \big)\,\d\cH^{n-1},
\end{align}
where $\div_{\de E} X := \sum_{i=1}^{n-1} g(\nabla_{\tau_i} X ,\tau_i)$ for a (measurable) local orthonormal frame 
$\{\tau_1, 
\ldots, \tau_{n-1}\}$ of $\de E$. (Note that in writing 
\eqref{e:destra} and \eqref{e:sinistra} we have used that $\de E$ is a $C^{1, \frac12}$ regular submanifold up to 
singular set of 
dimension at most $n-8$ and that $Y$ is supported in $B_\eta(y)$ where $\de E$ is smooth in order to make the 
integration by parts.)
In the case $V\in (0,\mu_g(M))$, we have $\rad(E)>0$ and thus $\P(E)<\infty$.
Moreover, from \eqref{eq:volumi} we deduce that
\begin{equation}\label{eq:XYvolume}
0=\frac{d}{dt}_{|t=0} \mu_g \left( \Phi^Y_t \circ  \Phi^X_t (E)  \right)= -\int_{\de E} g(X, 
\nu_E) \, \d\cH^{n-1} -  \int_{\de E} g(Y, \nu_E) \, \d\cH^{n-1}.
\end{equation}
Therefore, from \eqref{e:mean curvature const}, \eqref{e:destra}, \eqref{e:sinistra} and \eqref{eq:XYvolume} we conclude that
\begin{eqnarray}
\left\vert\int_{\de E} \div_{\de E} X\, \d\cH^{n-1}\right\vert &\leq&  \frac{1}{\rad(E)} \left(  \P(E) \|X\|_{\infty} + \left| \int_{\de E} g\big(\vec{H}_E, Y \big)\,\d\cH^{n-1} \right| \right)\nonumber\\
&\leq&   \frac{1}{\rad(E)} \left(  \P(E) \|X\|_{\infty} + |H_0|  \left| \int_{\de E} g(Y, \nu_E) \, \d\cH^{n-1} \right|   \right) \nonumber\\
&=&   \frac{1}{\rad(E)} \left(  \P(E) \|X\|_{\infty} + |H_0|  \left| \int_{\de E} g(X, \nu_E) \, \d\cH^{n-1} \right|  \right)  \nonumber\\
&\leq& C   \|X\|_{\infty}  \label{e:bounded H}
\end{eqnarray}
for some $C=C(\rad(E),\P(E),|H_0|)>0$, for every vector field $X$ with support contained in a metric ball $B_{\eta}(y)$ for some $y \in 
M$. By a simple partition of unity argument, \eqref{e:bounded H} holds for every $X \in \mathfrak{X}(M)$.
In particular, by the use of Riesz representation theorem we have proved the following lemma. To this regard we 
denote with 
$\mathcal{M} (M,TM)$ the vectorial Radon measures $\vec{\bf{\mu}}$ on $M$ with values in the tangent bundle $TM$.

\begin{lemma}[The mean curvature is represented by a vectorial Radon measure] \label{lem:LocFin1stVar} 
Let $E\subset M$ be an isoperimetric-isodiametric region for some $V\in (0,\mu_g(M))$
and denote by $B$ an enclosing ball. If $\de B$ is smooth, 
then there exists a vectorial radon measure $\vec{\bf{H}}_E\in \mathcal{M} (M,TM)$ concentrated on $\de E$ such that 
for every $C^1$ vector field  ${X}$ on $M$ with compact support, called $\Phi_t^{{X}}:M \to M$ the 
corresponding one-parameter family of diffeomorphisms for $t \in \R$, it holds 
\begin{equation}\label{eq:deltaE}
\delta E({X}):=\frac{d}{dt}_{|{t=0}} \P(\Phi_t^{{X}}(E))= - \int_M g({X}, \vec{\bf{H}}_E) .
\end{equation}
Moreover the total variation of $\vec{\bf{H}}_E$ is finite, i.e.  $$|\vec{\bf{H}}_E| (M) \leq C=C\big(\P(E), \rad(E), 
|H_0|\big)\in [0,\infty).$$
\end{lemma}

\begin{remark}
Note that
\begin{equation}\label{eq:bfEB}
\vec{\bf{H}}_E \llcorner B:= \vec{H}_E \,  \cH^{n-1}\llcorner (\partial E \cap B),
\end{equation}
where $\vec{H}_E$ is the mean curvature vector on the smooth part of $\de E$ as defined in \eqref{e:mean curvature 
const}.
\end{remark}

We close this subsection by noting that if
\begin{equation}\label{e:punto all'interno}
g\big(X(x),\nu_B(x)\big) \geq 0 \quad \forall\; x \in \de B \cap B_{\eta}(y),
\end{equation}
where $\nu_B$ is the interior normal to $\de B$ (note that $\de B \cap B_{\eta}(y)$ can also be empty),
then $\Phi^{{Y}}_t \circ \Phi^{{X}}_t (E) \subset B$ for $t\geq 
0$. In particular, the minimizing property of $E$ gives
\begin{equation}\label{e:riduce perimetro}
\P\big( \Phi^{{Y}}_t \circ \Phi^{{X}}_t (E)\big) \geq \P(E) \quad \forall \; t\geq 0,
\end{equation}
 which combined with \eqref{e:mean curvature const} and \eqref{eq:XYvolume} implies \fn
\begin{align}\label{e:1var-originale}
0&\leq \frac{d}{dt}\Big\vert_{t=0^+}  \P \fn \big( \Phi^{{Y}}_t \circ \Phi^{{X}}_t (E)  \big)
= \int_{\de E} \div_{\de E} X\, \d \cH^{n-1} - \int_{\de E} g\big(\vec{H}_E, Y \big)\notag\\
& = \int_{\de E} \div_{\de E} X\, \d \cH^{n-1} + H_0\int_{\de E} g\big( \nu_E \fn, X \big),
\end{align}
which in view of  \eqref {eq:deltaE} \fn  gives 
\begin{equation}\label{e:upper bound}
g\big(\nu_B, \vec{\bf{H}}_E\big) \llcorner (\de E \cap \de B) \leq H_0\,\cH^{n-1}\llcorner (\de E \cap \de B),
\end{equation}
where the inequality is intended in the sense of measures, i.e.~$ \int_A  g(\nu_B,\vec{\bf{H}}_E) \leq H_0 \cH^{n-1}(A)$ for 
every measurable set $A \subset \de E \cap \de B$.
%

\subsubsection{Orthogonality of $\vec{\bf{H}}_E$}
We have seen in the previous section that $\vec{\bf{H}}_E$ is well-defined as a measure on all $\de E$.
Translated into the language of varifolds,  we have shown that the integral varifold associated 
to $\de E$ has finite first variation.  A classical result due to Brakke \cite[Section 5.8]{Brakke}  (see also 
\cite{Menne} for an alternative proof and for fine structural properties of varifolds with locally finite first 
variation)  implies that for $\cH^{n-1}$-a.e. $x \in \de E$ it holds  $\vec{\bf{H}}_E (x)\in (T_x \partial E)^\perp$.  
This is not quite  enough  to our purposes, indeed in the next lemma we will show that $\vec{\bf{H}}_E$ is 
\emph{normal to $\de E$ as measure}, which is a strictly stronger statement.
Note that the proof is based on the fact that $E$ is a 
minimizer for the problem  \eqref{e:min problem},  and will not make use of the aforementioned  structural result by 
Brakke.

\begin{lemma}[The mean curvature measure is orthogonal to $\de E$]\label{lem:HENorm}
Let $E,B,M,V, \vec{\bf{H}}_E$ be as in Lemma \ref{lem:LocFin1stVar}. Then $\vec{\bf{H}}_E (x)\in (T_x \partial E)^\perp$ 
for  $|\vec{\bf{H}}_E|$-a.e. $x\in \partial E$, i.e.  the mean curvature is  orthogonal to $\partial E$ as a measure.
\end{lemma}

\begin{remark}
In other words there exists an $\R$-valued finite radon measure ${\bf H}_E$ on $M$ concentrated on $\de E$ such that 
$\vec{\bf H}_E={\bf H}_E \, \nu_E$; moreover, by \eqref{e:mean curvature const},
${\bf H}_E \llcorner (B \cap \de E)= H_0 \, \cH^{n-1} \llcorner (\de E \cap B)$. 
\end{remark}

\begin{proof}
In view of \eqref{e:mean curvature const} we only need to prove the claim for  $\vec{\bf{H}}_E \llcorner \de 
B$.
Assume by contradiction  that there exists  a compact subset $K\subset \de B \cap \de E$ such that 
\begin{equation}\label{eq:defK}
|\vec{\bf{H}}_E^T|(K)>0, 
\end{equation}
where $\vec{\bf{H}}_E^T:= P_{T\de E}(\vec{\bf{H}}_E)$ is the projection of $\vec{\bf{H}}_E$ onto the tangent space of 
$\de E$ (or, equivalently, onto $T\de B$, because $\de E$ and $\de B$ are $C^1$ and $T_x \de 
E=T_x \de B$ for every $x \in \de B\cap \de 
E$).

The geometric idea of the proof is very neat: if the mean curvature along $K\subset \de E \cap \de B$ has a non trivial 
tangential part, then deforming infinitesimally $E$ along this tangential direction will not increase the extrinsic 
radius (since the deformation of $E$ will stay in the ball $B$), will not increase the volume (because the deformation 
is tangential to $\de E$) but will strictly decrease the perimeter; so, after 
adjusting the volume in a smooth portion of $\de E$, this procedure builds   an infinitesimal deformation of $E$ which 
preserves the volume, does not increase the extrinsic radius but strictly decreases the   perimeter, contradicting that 
$E$ is a minimizer of the problem  \eqref{e:min problem}. The rest of the proof is  a technical implementation of this 
neat geometric idea.

For every $\varepsilon >0$ we construct a suitable $C^1$ regular tangential vector field.
To this aim, 
we consider the polar 
decomposition of the measure $\vec{\bf{H}}_E^T = v\,|\vec{\bf{H}}_E^T|$ where $v$ is a 
Borel vector field such that $v(x) \in T\de B$ and $g(v(x),v(x)) = 1$ for $|\vec{\bf{H}}_E^T|$-a.e.~$x \in M$.
By the Lusin theorem we can find a continuous vector field $w$ such that $|\vec{\bf{H}}_E^T|(\{v\neq w\}) \leq 
\varepsilon$ and $\supp(w) 
\subset K_\varepsilon : = \{ x \in \de E \cap \de B : d_g(x, K) < \eps\}$. Moreover, by a standard regularization 
procedure via mollification and projection on $T\de B$, 
we find a vector field $X_\varepsilon$ such that  $X_\varepsilon (x) \in T\de B$ for every $x \in \de B \cap K_{2\varepsilon}$,  
$\|X_\varepsilon - w\|_{\infty} \leq \varepsilon$ and $\supp(X_\varepsilon) 
\subset K_{2\varepsilon}$. \fn
Note that
\begin{align}\label{e:limite approx}
\int_M g\big(X_\varepsilon, \vec{\bf H}_E \big) & = 
\int_M g\big(X_\varepsilon - w, \vec{\bf H}_E \big) + \int_{\{w=v\}} g\big(v, \vec{\bf H}_E \big)
+  \int_{\{w\neq v\}} g\big(w, \vec{\bf H}_E \big)\notag\\
&\to |\vec{\bf{H}}_E^T|(K)\quad \text{as } \varepsilon \to 0.
\end{align}
Since $X_\varepsilon$ is a smooth vector field  compactly supported in $M$ and  tangent to $\de B$, the generated flow 
$\Phi^{{X}_\varepsilon}_t$ is well defined and  maps $B$ into $B$ for every $t \in \R$ and by \eqref{e:limite 
approx}
\begin{equation}
\frac{d}{dt}_{|t=0} \P \big(\Phi^{{X}_\varepsilon}_t (E) \big) = - \int_{\de E} g({X}_\varepsilon,  
\vec{\bf{H}}_E)\leq  - \frac{ |\vec{\bf{H}}_E^T|(K)}{ 2 \fn}<0,  
\label{eq:dtPPhiXeps}
\end{equation}
for $\varepsilon>0$ small enough. Moreover, since $X_\varepsilon$ is supported in $K_{2\varepsilon}$ and $K\subset \de 
B$ and $X_\varepsilon$ is tangent 
to $\de B = \de E$ in $K$, we have that
\begin{equation}\label{eq:dtmuPhiXeps}
\frac{d}{dt}_{|t=0} \mu_g \big(\Phi^{{X}_\varepsilon}_t (E) \big)= - \int_{\de E} g (\nu_E, {X}_\varepsilon) \, 
d \cH^{n-1} \to 0\quad \text{as } \varepsilon\to 0. 
\end{equation}

Up to choosing a smaller compact set, we can suppose that $K$  is 
contained in a small ball  $B_{r_0}(x)$ with $x \in \de E \cap \de B$ such that  $(\de E \setminus \de B) \cap (M\setminus 
B_{4r_0}(x)) \neq \emptyset$. \fn
Now fix $y \in \de E\setminus (\de B \cup  B_{4r_0}(x)\cup \sing(E))$ and let $r\in (0,r_0)$ be such that $B_{2r}(y) \cap 
(\de B \cup  B_{4r_0}(x)\cup \sing(E))=\emptyset$. 
For $\varepsilon>0$ small enough it is not difficult to construct a smooth vector field ${Y}_\varepsilon$ supported 
in $B_r(y)$ such that the generated flow $\Phi^{{Y}_\varepsilon}_t$ satisfies the following properties 
(\eqref{eq:PCmu} is intended for small $t$):
\begin{eqnarray}
\frac{d}{dt}_{|t=0} \mu_g(\Phi^{{Y}_\varepsilon}_t \circ \Phi^{{X}_\varepsilon}_t (E))&=&0 \label{eq:muPhiXY} \\
|\P(\Phi^{{Y}_\varepsilon}_t  (E), B_{2r}(y))- \P(E,  B_{2r}(y))|& \leq & C 
\mu_g(\Phi^{{Y}_\varepsilon}_t (E) \Delta E) \label{eq:PCmu}.
\end{eqnarray}
Notice that the combination of   \eqref{eq:dtmuPhiXeps}, \eqref{eq:muPhiXY} and  \eqref{eq:PCmu} gives
\begin{equation}\label{eq:dtPPhiY}
\left| \frac{d}{dt}_{|t=0} \P(\Phi^{{Y}_\varepsilon}_t  (E)) \right|  \leq C  \left| \frac{d}{dt}_{|t=0} 
\mu_g(\Phi^{{Y}_\varepsilon}_t  (E))  \right|=  C  \left| \frac{d}{dt}_{|t=0} \mu_g(\Phi^{{X}_\varepsilon}_t  (E)) 
 \right| \to 0, \text{ as } \varepsilon\to 0.
\end{equation}
Moreover, since for small $t>0$ we have $\Phi^{{Y}_\varepsilon}_t(E) \Delta E  \subset B_{2r}(y)$ which is disjoint 
from $\de B$,   and since by construction $\Phi^{{X}_\varepsilon}_t$ maps $B$ into $B$, it is clear that 
\begin{equation}\label{eq:dtradPhiY1} \nonumber
\Phi^{{Y}_\varepsilon}_t  \circ \Phi^{{X}_\varepsilon}_t (E) \subset B, \quad \text{for $t>0$ sufficiently small}. 
\end{equation} 
Therefore, since by assumption $E$ is a minimizer for  the problem   \eqref{e:min problem}, we infer
\begin{equation}\label{eq:dtradPhiY}
 \frac{d}{dt}_{|t=0}   \P(\Phi^{{Y}_\varepsilon}_t \circ \Phi^{{X}_\varepsilon}_t (E))  \geq 0. 
\end{equation} \fn
But on the other hand,  combining \eqref{eq:dtPPhiXeps} and  \eqref{eq:dtPPhiY}   we get
\begin{align}
\frac{d}{dt}_{|t=0}   \P(\Phi^{{Y}_\varepsilon}_t \circ \Phi^{{X}_\varepsilon}_t (E)) \; 
& =& 
 \frac{d}{dt}_{|t=0} \P(\Phi^{{Y}_\varepsilon}_t  (E)) +  \frac{d}{dt}_{|t=0} 
\P(\Phi^{{X}_\varepsilon}_t  (E))  \nonumber\\
& \leq&  - \frac{ |\vec{\bf{H}}_E^T|(K)}{ 4 \fn}<0, \text{  for $\varepsilon>0$ small enough.}  \nonumber
\end{align}
Clearly the last inequality contradicts \eqref{eq:dtradPhiY}. We conclude that it is not possible to find a compact subset $K \subset \de B \cap \de E$ satisfying \eqref{eq:defK}; 
therefore the measure $|\vec{\bf{H}}_E^T|$ vanishes identically and the proof is complete.  
\end{proof}

\subsubsection{$L^\infty$ estimate}

The next step is to show that the signed measure ${\bf H}_E$ is actually absolutely continuous with respect to 
$\cH^{n-1}\llcorner \de E$ with $L^\infty$ bounds on the density. The upper bound follows from \eqref{e:upper bound}.
For the lower bound we use the following lemma which is an adaptation of \cite[Theorem 2]{WhiteMaxPrinc} to our 
setting (notice that the statement of 
\cite[Theorem 2]{WhiteMaxPrinc} is more general as includes higher co-dimensions and arbitrary varifolds, but let us 
state below just the result we will use in the sequel).

\begin{lemma}\label{lem:maxprinc}
Let $N^n \subset M^n$ be an $n$-dimensional submanifold with $C^2$-boundary  $\de N$ and denote with ${\nu}_N$  the 
inward pointing unit normal to $\de N$. Fix a compact subset $K\subset \de N$ and  assume that, denoted with $\vec{H}_N$ 
the mean curvature of $\de N$, it holds
\begin{equation}\nonumber
g(\vec{H}_N, {\nu}_N) \geq \eta, \quad \text{on } K.
\end{equation}
Then, for every $\varepsilon>0$ there exists a $C^1$-vector field ${X}_\varepsilon$  on $M$  with the following 
properties:
\begin{eqnarray}
{X}_\varepsilon(x)&=&{\nu}_N, \quad \forall x \in K  \label{eq:XepsnuN} \\ 
|{X}_\varepsilon| (x) &\leq& 1   , \quad \forall x \in M  \label{eq:|Xeps|<1} \\
\supp({X}_\varepsilon) &\subset& K_\varepsilon:= \{x \in M\,:\, d(x,K)\leq \varepsilon\} \label{eq:suppXe} \\
g({X}_\varepsilon, {\nu}_N) (x) &\geq& 0,  \quad \forall x \in \de N  \label{eq:gXenuN}, \\
\frac{d}{dt}_{|t=0} \P(\Phi^{{X}_\varepsilon}_t (E)) &\leq& -\eta \int_{\de E} |{X}_\varepsilon| \, d \cH^{n-1}, 
\label{eq:Xe1stvar} 
\end{eqnarray}
for every subset $E\subset N$ with $C^1$ boundary $\de E$, where $\Phi^{{X}_\varepsilon}_t$ denotes the flow 
generated by the vector field ${X}_\varepsilon$.
\end{lemma}

Lemma \ref{lem:maxprinc} will be used to prove the following lower bound on the mean curvature measure ${\bf H}_E$ of 
$\de E$.

\begin{lemma}[Lower bound on ${\bf H}_E$]\label{lem:LBHE}
Let $E,B,M,V, \vec{\bf{H}}_E, {\bf H}_E$ be as in Lemma \ref{lem:HENorm}. Assume  $\eta:= \inf_{\de 
B} H_B>-\infty$, where $H_B:=g(\vec{H}_B, {\nu}_B)$ and $\vec{H}_B$ 
is the mean curvature vector of $\de B$. Then
\begin{equation}\label{eq:LBHE}
{\bf{H}}_E \llcorner (\de E \cap \de B) \geq \eta \, \cH^{n-1} \llcorner    (\de E \cap \de B).
\end{equation}
\end{lemma}

\begin{proof}

Fix any $K \subset \de E \cap \de B$.
For every $\varepsilon\in(0,1)$ let ${X}_\varepsilon$ be the $C^1$ vector field obtained by applying Lemma  
\ref{lem:maxprinc} with $N=B$, then by \eqref{eq:Xe1stvar} and \eqref{eq:suppXe} we get
\begin{eqnarray}
-\eta \int_{\de E} |{X}_\varepsilon| \, d \cH^{n-1} &\geq&  \frac{d}{dt}_{|t=0} \P(\Phi^{{X}_\varepsilon}_t 
(E))=- \int_{K_\varepsilon} g({X}_\varepsilon, {\nu}_E) \, d {\bf{H}}_E \nonumber \\
&=&  - \int_{K} g({X}_\varepsilon, {\nu}_B) \, d {\bf{H}}_E  - \int_{K_\varepsilon \setminus K} 
g({X}_\varepsilon, {\nu}_E) \, d {\bf{H}}_E\notag\\
&& \to - {\bf{H}}_E (K), \quad\text{ as $\varepsilon \to 0$,}  \label{eq:-etaint>}
\end{eqnarray}
where in the second identity we used that  $\nu_B=\nu_E$ on  $K\subset \de E \cap \de B$.
Using  \eqref{eq:XepsnuN} and \eqref{eq:|Xeps|<1}, we have
\begin{eqnarray}
-\eta \int_{\de E} |{X}_\varepsilon| \, d \cH^{n-1} &=&  -\eta \int_{K} |{X}_\varepsilon| \, d \cH^{n-1} -\eta 
\int_{\de E \cap (K_\varepsilon\setminus K)} |{X}_\varepsilon| \, d \cH^{n-1}  \nonumber \\
& &  \to  -\eta \, \cH^{n-1}(K) \quad\text{ as $\varepsilon \to 0$.} \label{eq:-etaint<}
\end{eqnarray}
In particular, in the limit as $\varepsilon \to 0$ we deduce from \eqref{eq:-etaint>} that
\begin{equation}\label{eq:etaHn-1>etadelta}
\eta\,  \cH^{n-1} (K) \leq {\bf{H}}_E (K).  
\end{equation}
Since this holds for every $K \subset \de E\cap \de B$, it is easily recognized that \eqref{eq:LBHE} follows.
\end{proof}

\subsection{Optimal regularity}\label{s:Opt}

In this section we prove that the boundary of an isoperimetric-isodiametric set $E$ is $C^{1,1}$ regular away from the singular set.

\begin{theorem}\label{p:regularity2}
Let $E\subset M$ be an isoperimetric-isodiametric set and $x_0 \in M$ be such that $\mu_g(E \setminus 
B_{\rad(E)}(x_0)) = 0$. Assume that $B:= B_{\rad(E)}(x_0))$ has smooth boundary.
Then, there exists $\delta>0$ such that $\de E \setminus B_{\rad(E) - \delta}(x_0)$ is $C^{1,1}$ regular.
\end{theorem}

Note that the $C^{1,1}$ regularity is optimal, because in general one cannot expect to have continuity of the second 
fundamental form of $\de E$ across the free boundary of $\de E$, i.e.~the points on the relative (with respect to $\de 
B$) boundary of $\de E \cap \de B$.

\subsubsection{Co-ordinate charts}\label{s:co-ordinates}
We start fixing suitable co-ordinate charts.
Since $E$ is bounded, there exists $r_0>0$ such that for every $x_0 \in \de 
E$ there is a normal co-ordinate chart $(\Omega, \varphi)$ with $x_0 \in \Omega$ and
\[
\varphi: \Omega \subset M \to B^{n-1}_{r_0} \times (-r_0,r_0) \subset \R^{n-1} \times \R
\]
such that $\varphi(x_0) = 0$, $g(0) = \Id$ and $\nabla g(0) = 0$, where $g$ denotes the metric tensor in these 
co-ordinates.
Moreover, by the $C^{1,\frac12}$ regularity of $\de E$ established  in \S~\ref{s:C1alpha}, up to rotating these co-ordinate chart and 
eventually changing $r_0$, we can also assume that for every point $x_0 \in \de B \cap \de E$ also the following holds:
\begin{itemize}
\item $\de E$ and $\de B$ are, respectively, $C^{1, \frac12}$ and $C^{\infty}$ regular submanifolds,  given in this chart as 
graphs of functions $u, \psi: B^{n-1}_{r_0} \to \left(-\frac{r_0}{2}, \frac{r_0}{2}\right) \fn$ with $u \in C^{1,\frac12}$ and $\psi \in C^{\infty}$;

\item the functions $u$ and $\psi$ satisfy $\psi(x) \leq u(x)$ for every $x \in B^{n-1}_{r_0}$,
\[
u(0) = \psi(0) = |\nabla u(0)| = |\nabla \psi(0)| = 0, 
\]
and $\|u\|_{C^1}\leq \delta_0$ and $\|\psi\|_{C^1}\leq \delta_0$
for a fixed  $\delta_0>0$ which will be later assumed to be suitably small.
\end{itemize}

\medskip 

On every such a chart, the $C^{1,\frac12}$ regular submanifold $\de E \cap \Omega$ is given
as the set $\{(x, u(x)) : x \in B^{n-1}_{r_0} \}$.
We can consider the natural co-ordinate chart on it given by $(x,u(x)) \mapsto x \in B_r^{n-1}$ with induced metric tensor given by $h_{ij} := g(E_i, E_j)$, where $E_i := e_i + \de_i u \, e_n$ 
for $i=1,\ldots, n-1$.
In particular, 
\begin{equation}\label{e:metrica sul grafico}
 h_{ij} = g_{ij} + \de_i u \, g_{nj}+\de_ju \,g_{ni} + \de_i u\,\de_j u \, g_{nn}, \fn
\end{equation}
where $\de_iu = \de_iu(x)$ and $g_{ij} = g_{ij}(x,u(x))$. We will use the notation  $\tilde{h}$ for the function
$\tilde{h}:B^{n-1}_{r_0} \times \R \times \R^n \to \R^{n\times n}$ 
$$\tilde{h}_{ij}(x,z, p) = g_{ij}(x,z) + p_i g_{jn}(x,z) + p_j g_{ni}(x,z) + p_i \, p_j g_{nn}(x,z)$$
with the obvious relation that $h_{ij} = \tilde{h}_{ij}\big(x,u(x), \nabla u(x)\big)$. Note that, as a function in $(x,z, p)$, $\tilde{h}$ is smooth.

\subsubsection{First variation formula in local co-ordinates}
We consider next functions $\phi \in C^\infty_c(B^{n-1}_{r_0})$
and $\chi \in C^\infty_c(-r_0,r_0)$, and we assume that $\chi\vert_{(-\frac{r_0}{2},\frac{r_0}{2})} \equiv 1$, in such 
a way to assure that $\chi \circ u(x) =1$ for every $x \in B_{r_0}^{n-1}$ (by the assumptions made on $u$).
Consider the associated vector field $X (x,y) := \phi(x)\,\chi(y) \, e_n$ and note that $X \in 
C^{\infty}_c(\Omega,\R^n)$ and $X\vert_{\de E} = \phi(x)\, e_n$.
Called $F(t, p) := p + t\, X(p)$, there exists $\epsilon_0>0$ such that $F_t:= F(t, \cdot)$ is a diffeomorphism of $\Omega$ into itself for every $|t| 
\leq \epsilon_0$. 

 Consider now the variations of the
area along these one-parameter family of diffeomorphisms under the assumption $\phi\geq0$ on $\Lambda(u):=\{x\in B^{n-1}_{r_0} \,:\, u(x)=\psi(x)\}$.  \fn Arguing as in \eqref{e:1var-originale},  we get that
\begin{align}\label{e:1var}
0&\leq \int_{\de E} \div_{\de E} X\, \d \cH^{n-1}  - \fn H_0 \int_{\de E} g(X, \nu_E)\,\d\cH^{n-1}\notag\\
& = \int_{\Sigma} h^{ij}g(\nabla_{E_i}X, E_j)\,\d \cH^{n-1}  - \fn H_0 \int g(X, \nu_E)\,\d \cH^{n-1},
\end{align}
where in the second line we have used a simple computation for the tangential divergence of $X$.
Noting that
\begin{align*}
\nabla_{E_i} X & = \nabla_{e_i + \de_i u\, e_n} X = \nabla_{e_i}X + \de_i u\, \nabla_{e_n}X\\
& = \de_i\phi\,e_n + \phi\,\nabla_{e_i}e_n + \de_i u\, \phi \,\nabla_{e_n} e_n\\
& = \de_i\phi\,e_n + \phi\,\Gamma_{in}^k\,e_k + \de_i u\, \phi\, \Gamma_{nn}^k\,e_k,
\end{align*}
we get that
\begin{align}
h^{ij}\,g(\nabla_{E_i}X, E_j) & = h^{ij}\,\big(\de_i\phi\,g_{jn} + \phi\,\Gamma_{in}^k\,g_{jk} + \de_i u\, \phi\, 
\Gamma_{nn}^k\,g_{jk}\big)\notag\\
& \quad + h^{ij}\,\big(\de_ju\,\de_i\phi\,g_{nn} + \phi\,\de_ju\,\Gamma_{in}^k\,g_{kn} + \de_ju\,\de_i u\, \phi\, 
\Gamma_{nn}^k\,g_{kn}\big)\notag\\
& = \de_i\phi\,\big(h^{ij}\,g_{jn} + h^{ij}\,\de_ju\,g_{nn}\big)\notag\\
&\quad+ \phi\,\big(h^{ij}\,\de_i u\, \Gamma_{nn}^k\,g_{jk} + h^{ij}\,\de_ju\,\de_i u\, \Gamma_{nn}^k\,g_{kn})\notag\\
&\quad + \phi\,\big(h^{ij}\,\Gamma_{in}^k\,g_{jk} + h^{ij}\,\de_ju\,\Gamma_{in}^k\,g_{kn}\big).
\end{align}

In particular, by a simple integration by parts, \eqref{e:1var} reads as
\begin{equation}\label{e:1var-2versione}
\int_{B_r^{n-1}} \phi\,Lu\, \sqrt{\det(h_{ij})}\,\d x \leq 0 \quad
\forall\;\phi \in C_c^{1}(B_r^{n-1}), \; \phi\vert_{\Lambda(u)} \geq 0,
\end{equation}
where $\Lambda(u) := \{x \in B_r^{n-1}: u(x) = \psi(x) \}$ and
\begin{equation}\label{e:operatore}
Lu (x) := \div \big(A(x, u(x), \nabla u(x)) \nabla u(x) + b(x, u(x), \nabla u(x)) \big) - f(x)
\end{equation}
with
\begin{itemize}
\item $A = (a^{ij})_{i,j=1,\ldots, n-1} : B_r^{n-1} \times (-r,r) \times \R^{n-1} \to \R^{(n-1)\times (n-1)}$ is a 
smooth function given by
\[
a^{ij}(x, z, p) := g_{nn}(x,z)\,\tilde{h}^{ij}(x, z, p);
\]

\item $b: B_r^{n-1}\times (-r,r) \times \R^{n-1} \to \R^{n-1}$ is a smooth regular function given by
\[
b^i(x,z,p) := \tilde{h}^{ij}(x, z, p) \,g_{jn}(x,z);
\]

\item $f: B_r^{n-1} \to \R$ is a $C^{0,\alpha}$ regular function given by
\begin{align*}
f(x) & := h^{ij}\,\de_i u\, \Gamma_{nn}^k\,g_{jk} + h^{ij}\,\de_ju\,\de_i u\, \Gamma_{nn}^k\,g_{kn}\\
& \quad + h^{ij}\,\Gamma_{in}^k\,g_{jk} + h^{ij}\,\de_ju\,\Gamma_{in}^k\,g_{kn}  - \fn H_0 \,g\big(e_n, \nu_E\big),
\end{align*}
where $h^{ij} = \tilde{h}^{ij}\big(x,u(x), \nabla u(x)\big)$, $g_{ij}= g_{ij}\big(x, u(x)\big)$, $\Gamma_{ij}^k = 
\Gamma_{ij}^k\big(x,u(x) \big)$ and $\nu_E = \nu_E(x, u(x))$.
\end{itemize}

Explicitly expanding the divergence term in $Lu$ we deduce that
\begin{align}
Lu(x) = c^{ij} \de_{ij} u + d,
\end{align}
where 
\begin{equation}\label{e:cij}
c^{ij} = a^{ij} + g_{nn}\,\de_l u\, \de_{p^j}h^{il} + g_{ln}\, \de_{p_j}h^{il},
\end{equation}
with $\de_{p^j}h^{il} = \de_{p^j} \tilde{h}^{il}\big(x,u(x), \nabla u(x)\big)$,  $g_{ij}= g_{ij}\big(x, u(x)\big)$ \fn and $d \in 
C^{0,\alpha}(B_r^{n-1})$ is given by
\begin{align}\label{e:d}
d & = g_{nn}\de_i h^{ij} \de_j u +   g_{nn} \fn \de_z h^{ij} \de_i u\, \de_j u+ \de_i g_{nn} h^{ij} \de_j u + \de_n g_{nn} 
h^{ij}\de_i u\, \de_j u\notag\\
& \quad+ 
g_{jn}\de_i h^{ij} + g_{jn}\de_z h^{ij} \de_i u + \de_i g_{jn} h^{ij} + \de_n g_{jn}h^{ij} \de_i u - f
\end{align}
with the entries of $h$ and of its derivatives are computed in $\big(x,u(x), \nabla 
u(x)\big)$, while those of $g$ and the derivatives of the metric are computed in  $\big(x, u(x)\big)$.

Note that \eqref{e:1var-2versione} is equivalent to the following couple of differential relations:
\begin{equation}\label{e:1var-2versione-1}
\begin{cases}
Lu  \leq \fn 0 & \text{in } B_r^{n-1},\\
Lu =0 & \text{in } B_r^{n-1}\setminus \Lambda(u),
\end{cases}
\end{equation}
where the first inequality is meant in the sense of distribution, while the second equation is pointwise (also 
recalling that $u$ is smooth outside the contact set $\Lambda(u)$).

\subsubsection{Quadratic growth}

Note that by the explicit expressions of the previous subsection it turns out that $c^{ij}, d \in 
C^{0,\alpha}(B_{r_0}^{n-1})$ with uniform estimates (by the assumptions in \S~\ref{s:co-ordinates})
\begin{equation}\label{eq:cdC0alpha}
\|c^{ij}\|_{C^{0,\alpha}(B_{r_0}^{n-1})} + \|d\|_{C^{0,\alpha}(B_{r_0}^{n-1})} \leq C.
\end{equation}
Since $c(0) = \Id$ and $c^{ij}$ are  H\"older continuous, up to choosing a smaller $\delta_0>0$ (and consistently a smaller $r_0>0$) we can also ensure that $c^{ij}$ is 
uniformly elliptic with bounds
\[
\frac{\Id}{2}  \leq c \leq 2\, \Id.
\]

The next lemma shows that $u$ leaves the obstacle $\psi$ at most as a quadratic function of the distance to the
free-boundary point.

\begin{proposition}\label{p:quadratic}
Let $E \subset M$ be an isoperimetric-isodiametric set. Then, there exists a constant $C>0$ such that,
for every $x_0 \in \de E \cap \de B$, setting co-ordinates as in \S~\ref{s:co-ordinates}, we have that
\begin{equation}\label{e:quadratic}
u(x) - \psi(x) \leq C \, |x|^2 \quad \forall\; x \in B_{\frac{r_0}{2}}^{n-1}.
\end{equation}
\end{proposition}

\begin{proof}
Let us consider the homogeneous part of the operator $L$, i.e.~$\cL w := c^{ij} \de_{ij} w$.
Since $\cL (u - \psi) = Lu - \cL \psi - d$, for every $r  \leq \fn r_0$ we can write $(u - \psi)\vert_{B_r^{n-1}} = w_1 + w_2$ 
with
\begin{equation}
\begin{cases}
\cL w_1 = 0 & \text{in } B_{r}^{n-1},\\
w_1 = u - \psi & \text{on } \de B_{r}^{n-1},
\end{cases}
\end{equation}
and 
\begin{equation}
\begin{cases}
\cL w_2 = L u - \cL \psi - d& \text{in } B_{r}^{n-1},\\
w_2 = 0 & \text{on } \de B_{r}^{n-1}.
\end{cases}
\end{equation}

We start estimating $w_2$ from below.
Considering that $    \cL w_2 + \cL \psi + d = Lu \leq  0 \fn$, we can apply the $L^\infty$-estimate for elliptic equations 
\cite[Theorem 8.16]{GT}. In order to understand 
the dependence of the constant on the domain, we can rescale the variables in this way: $v:B_1^{n-1} \to \R$ given by 
$v(y) := r^{-2} w_2(r\,y)$.
Then, the equation satisfied by $v$ is
\[
\cL v(y) + \cL \psi(ry) + d(ry)= L u (ry)  \leq \fn 0.
\]
We can then conclude using \cite[(8.39)]{GT} that
\[
 \sup_{B_1^{n-1}} (-v) \fn \leq C\, \| \cL \psi(ry) + d(ry) \|_{L^{\frac{q}{2}}(B_1^{n-1})} \leq C,
\]
where now $C$ is a dimensional constant (only depending on $q>n-1$, which for us is any fixed exponent --
note that the hypothesis (8.8) in \cite[Theorem 8.16]{GT} is satisfied because we are considering the operator $\cL$ which has no lower order terms).
In particular, scaling back to $w_2$ we deduce that
\begin{equation}\label{e:w_2 from below}
w_2(x) \geq - C\,r^2 , \quad \forall\; x \in B_r^{n-1}.
\end{equation}

This clearly implies that $w_1(0) = u(0) -\psi(0) \fn - w_2(0) \leq C\,r^2$.
We can then use Harnack inequality for $w_1$ (cf. \cite[Theorem 8.20]{GT}) and conclude that
\begin{equation}\label{e:w_1 from above}
w_1(x) \leq C\,\inf_{B_{\frac{r}{2}}^{n-1}} w_1 \leq C\,w_1(0) \leq C\,r^2 , \quad \forall\; x \in B_{\frac{r}{2}}^{n-1}. \fn
\end{equation}

Finally note that in $B_r^{n-1}\setminus \Lambda(u)$ we have the equality
$\cL w_2 = - \cL \psi - d$. Therefore, the function
$z:= w_2 + C\,|x|^2$ satisfies $\cL z \geq 0$ for a suitably chosen constant $C = C(\|\cL \psi\|_{L^\infty}, 
\|d\|_{L^\infty}) $.
By the strong maximum principle \cite[Theorem 8.19]{GT} we deduce that
\[
\max_{B_r^{n-1}\setminus \Lambda(u)} z \leq \max_{\de (B_r^{n-1}\setminus \Lambda(u))} z
\leq C\, r^2,
\]
where we used that $z\vert_{\de B_r^{n-1}} = C\,r^2$ and that for every $x \in {\Lambda(u)}\cap B_r^{n-1}$ we have 
$z(x)= - w_1(x) + C\, |x|^2 \leq C\, r^2$ by the positivity of $w_1$.
In conclusion, we have that $u(x) - \psi(x) \leq |w_1(x)| + |w_2(x)| \leq C\,r^2$ for every $x \in 
B_{\frac{r}{2}}^{n-1}$.
Since $r \leq \fn r_0$ is arbitrary, by eventually changing the constant $C$ we conclude the proof of the proposition.
\end{proof}

\subsubsection{Curvature bounds away from the contact set}
Next we analyze the points $p \in \de E \setminus \de B$ which are close to $\de B$.
To this aim we fix a constant $s_0>0$ such that the following holds: if $\dist(p, \de 
E \cap \de B) = \dist(p,x_0) < s_0$, then $p$ belongs to the co-ordinate chart $\Omega$ around $x_0$ as fixed in 
\S~\ref{s:co-ordinates} and moreover, in these co-ordinates, $p = (x,z) \in B_{r_0}^{n-1} \times (-r_0, r_0)$
(necessarily with $x \not\in \Lambda(u)$) satisfies
\[
B_{4\delta}^{n-1}(x) \subset B_{r_0}^{n-1}  
\quad \text{with}\quad \delta:= \frac{\textrm{dist}(x, \Lambda(u))}{2}.
\]
Note that the existence of such a constant $s_0>0$  is ensured \fn by a simple compactness argument.
Recall also that by the quadratic growth proved in the previous section we now that
\[
\|u\|_{L^\infty(B_{2\delta}^{n-1}(x))} \leq C\,\delta^2.
\]

The following lemma gives a curvature bound for $\de E$ in points $p$ as above.

\begin{lemma}\label{l:stima D2}
Let $p \in \de E \setminus \de B$ satisfy $\dist(p, \de E \cap \de B) < s_0$.
Fixing $x_0 \in \de E \cap \de B$ and the corresponding co-ordinate chart as in \S~\ref{s:co-ordinates} with 
the notation fixed above, we then conclude that
\begin{equation}\label{e:curvature}
\|D^2 u\|_{L^\infty(B_{\delta}^{n-1}(x))} \leq C,
\end{equation}
where $C>0$ is a dimensional constant.
\end{lemma}

\begin{proof}
 Since on $B_{4\delta}^{n-1}\subset B^{n-1}_{r_0} \setminus \Lambda(u)$ the \emph{equation} $Lu=0$ is satisfied, \fn
the proof is a consequence of the basic interior Schauder estimates for second order elliptic 
equations (cp.~\cite[Theorem~6.2]{GT}).  More precisely we write the equation as $\cL u=-d$ where $d \in C^{0,\alpha}$ was defined is  \eqref{e:d} and satisfies \eqref{eq:cdC0alpha}, and we apply  \cite[Theorem~6.2]{GT}) to such an equation. \fn
Indeed, by simply recalling the definition of the norms in \cite[Theorem~6.2]{GT} we have that,
setting $\sfd_y:= \dist(y,\de B_{2\delta}^{n-1}(x))$
\begin{align}
\delta^2\,\|D^2 u\|_{L^\infty(B_{\delta}^{n-1}(x))} & \leq C\, \left( \|u\|_{L^\infty(B_{2\delta}^{n-1}(x))} +
\sup_{y\in B_{2\delta}^{n-1}(x)}\sfd_y^2 |d(y)| \right)\notag\\
& \quad + C\,\sup_{y,z\in B_{2\delta}^{n-1}(x)}\min\left\{\sfd_y,\sfd_z\right\}^{2+\alpha} 
\frac{|d(y)-d(z)|}{|y-z|^\alpha}\notag\\
& \leq C\, \left( \|u\|_{L^\infty(B_{2\delta}^{n-1}(x))} + \delta^2 
\|d\|_{L^\infty(B_{2\delta}^{n-1}(x))}\right)\notag\\
&\quad+ C\,\delta^{2+\alpha} [d]_{C^{0,\alpha}(B_{2\delta}^{n-1}(x))} \leq C\, \delta^2.\notag\qedhere
\end{align}
\end{proof}

\subsubsection{$C^{1,1}$-regularity}

In this section we finally prove  Theorem~\ref{p:regularity2}. 
The proof is based on the following property: by Proposition~\ref{p:quadratic} and Lemma~\ref{l:stima D2}, there exists 
$\delta>0$ such that for every $ x_0 \fn \in \de B\cap \de E$ there exists   $r_{0}>0$ \fn satisfying the following: fixing 
co-ordinates as in \S~\ref{s:co-ordinates},  
\begin{equation}\label{eq:uxyquad0}
|u(y)-u(x)-\nabla u(x) \cdot (y-x)| \leq \frac{\bar{C}}{2} \; |x-y|^2, \quad \forall \;x, y \in B_{r_{0}}(x_0).
\end{equation}  \fn
Indeed, if $x \in \de E \cap \de B$, then  centering the  co-ordinates at  $x$ we have  $0=u(0) = |\nabla u(0)|$, and 
\eqref{eq:uxyquad0} is a direct consequence of \eqref{e:quadratic}. On the other hand, if  $x \notin \de E \cap \de B$, \fn 
then setting the co-ordinates as in Lemma~\ref{l:stima D2}, we deduce \eqref{eq:uxyquad0} from \eqref{e:curvature}.

The conclusion of Theorem~\ref{p:regularity2} is then a direct consequence of the following lemma combined with a standard partition of unity argument. 

\begin{lemma}\label{l:distrib}
Let $\Omega \subset \R^n$ be an open subset and let  $u:\Omega \to \R$ be a $C^1$-function. Assume there exist 
$\bar{C}>0$ and a  countable covering $\{B_i\}_{i \in \N}$ of $\Omega$ made  by open balls $B_i\subset \Omega$ such that 
for every $x,y \in B_i$ it holds 
\begin{equation}\label{eq:uxyquad}
|u(y)-u(x)-\nabla u(x) \cdot (y-x)| \leq \frac{\bar{C}}{2} \; |x-y|^2.
\end{equation}
Then the distribution $\de^2_{ij} u\in {\mathcal D}'(\Omega)$ is represented by an $L^{\infty}(\Omega)$ function, and 
$$\|\de ^2_{ij} u\|_{L^\infty(\Omega)} \leq \bar{C}.$$
\end{lemma}

\begin{proof}
By  a standard partition of unity argument   it is enough to prove that for every ball $B_i$ the restriction of the 
distribution $\de^2_{ij} u \llcorner B_i$ is represented by an $L^{\infty}(B_i)$ function, and $\|\de ^2_{ij} 
u\|_{L^\infty(B_i)} \leq \bar{C}.$
In order to simplify the notation let us fix $i\in \N$ and denote $B:=B_i$.
For every fixed $\varphi\in C^\infty_c(B)$ let $Q^{\varphi}:\R^n \times \R^n \to \R$ be defined by 
\begin{equation}\label{eq:defQphi}
Q^\varphi(v_1,v_2):=\int_{B} u \; \frac{\de^2 \varphi}{\de v_1 \de v_2}  \; .
\end{equation}
We first claim that 
\begin{equation}\label{eq:estQphi}
|Q^{\varphi}(v,v)| \leq \bar{C} |v|^2 \|\varphi\|_{L^1(B)}, \quad \forall \varphi\in C^\infty_c(B),\; \forall v\in \R^n,
\end{equation}
where $\bar{C}$ is given is \eqref{eq:uxyquad}.  In order to prove \eqref{eq:estQphi}, we write  \eqref{eq:uxyquad} 
exchanging $x$ and $y$ and sum up to get 
\begin{equation}\nonumber 
|(\nabla u(x)- \nabla u(y)) \cdot (x-y)| \leq \bar{C} \; |x-y|^2.
\end{equation}
Choosing $y=x+tv$ in the last estimate,  we get
\begin{equation} \label{eq:nablauxy}
\frac{|(\nabla u(x+tv)- \nabla u(x)) \cdot v |}{t} \leq \bar{C}, \quad \forall v\in S^{n-1}, \; \forall  t\in (0, 
1-|x|).
\end{equation}
Now using that $u$ is $C^1$ and  $\varphi\in C^\infty_c(B)$, we can integrate by parts to get
\begin{eqnarray}
\left| \int_{B} u  \; \frac{\de^2 \varphi}{\de v \de v}  \right| &=& \left| \int_{B} \frac{ \de u}{\de v} \; \frac{\de 
\varphi}{\de v}  \right| = \left| \int_{B} (\nabla u(x) \cdot v) \;  \lim_{t\downarrow 0} \frac{\varphi(x+tv)- 
\varphi(x)}{t} \, dx \right| \nonumber \\
&=&   \left|  \lim_{t\downarrow 0} \int_{B} \left( \frac{\nabla u(x-tv)-\nabla u(x) }{t}\cdot v \right) \; \varphi(x)\;  
dx  \right|  \nonumber \\
& \leq &\bar{C} \, \|\varphi\|_{L^1(B)}, \quad \forall v \in S^{n-1}, \label{eq:pfClaim1}
\end{eqnarray}
where in the second line we used the change of variable $x\mapsto x+tv$, and the last inequality follows from  
\eqref{eq:nablauxy}. The inequality \eqref{eq:pfClaim1} proves our claim \eqref{eq:estQphi}.
\\We now show that \eqref{eq:estQphi} implies that  the distribution $\de^2_{ij} u$ is represented by an $L^{\infty}(B)$ 
function and $\|\de ^2_{ij} u\|_{L^\infty(B)} \leq \bar{C}.$
To this aim observe that for every $\varphi \in C^\infty_c(B)$, by the Schwartz lemma,  the map  $Q^\varphi:\R^n \times 
\R^n \to \R$ defined in \eqref{eq:defQphi} is a symmetric bilinear form. Using \eqref{eq:estQphi}, by polarization of 
$Q^\varphi$ we infer
\begin{equation}\label{eq:polarizQ}
|Q^\varphi(\de_i, \de_j)|=\frac{1}{4} |Q^{\varphi}(\de_i+\de_j,\de_i+\de_j) -  Q^{\varphi}(\de_i-\de_j,\de_i-\de_j) | 
\leq \bar{C} \, \|\varphi\|_{L^1(B)},
\end{equation}
for every $ i,j=1,\ldots,n$. But now
$$Q^\varphi(\de_i, \de_j)= <\partial^2_{ij} u, \varphi>_{{\mathcal D}', {\mathcal D}} \quad , $$
where $<\cdot , \cdot>_{{\mathcal D}', {\mathcal D}}$ denotes the pairing between  distributions and $C^\infty_c$-test 
functions. Therefore  \eqref{eq:polarizQ} combined with Riesz representation Theorem  concludes the proof.
\end{proof}

The arguments above prove  also the  following slightly more general regularity result for isoperimetric regions inside a $C^{2}$-domain. In order to state it, for a subset $A\subset M$ and for some $\delta>0$, let us denote with $B_{\delta}(A)=\{x\in M \,:\, \inf_{y\in A} \sfd(x,y) \leq \delta \}$ the $\delta$-tubular neighborhood of $A$.

\begin{theorem}[$C^{1,1}$-regularity of isoperimetric regions inside a $C^{2}$-domain]\label{thm:regIsopObst}
Let $(M,g)$ be a Riemannian manifold, let $\Omega \subset M$ be an open subset with $C^{2}$ boundary $\partial \Omega$ and fix $v \in (0,\mu_{g}(\Omega))$. Let $E\subset \Omega$ be a finite perimeter set with $\mu_{g}(E)=v$ and minimizing the perimeter among regions contained in $\Omega$, i.e.
$$
\Pe(E)= \inf \{ \Pe(F) \,:\, F\subset \Omega,\; \mu_{g}(F)=v \}.
$$
Then, there exists $\delta>0$ such that $\de E \cap B_{\delta}(\de \Omega)$ is $C^{1,1}$ regular.
\end{theorem} 

{
\begin{remark}\label{r:gerhardt}
Theorem \ref{thm:regIsopObst} already appeared in \cite[Proposition, Pag. 418] {WhiteJDG91}. On the other hand, the arguments in the proof of \cite[Proposition, Pag. 418] {WhiteJDG91} are very concise (line 7, pag. 419 in \cite{WhiteJDG91})
and basically consist in referring to  the work of Gerhardt \cite{Gerhardt}. Nevertheless, it seems that  one of the hypotheses of  \cite{Gerhardt}
is not met for the operator $H$ in \cite{WhiteJDG91}. Indeed,   $H$ is
the Euler-Lagrange operator of the functional 
\[
\Phi(u) = \int L(x, u(x), \nabla u(x)) dx,
\]
and a simple computation shows that
\[
H(u) = \frac{\partial L}{\partial z} (x, u(x), \nabla u(x)) - {\rm div} \Big(
\frac{\partial L}{\partial p} (x, u(x), \nabla u(x))\Big),
\]
where we named the variables as  $L = L(x,z,p)$.
Now the operator $H$ is of the form considered in  \cite{Gerhardt}
(here there is a conflict of notation between the two papers,
therefore we put a bar for the notation in \cite{Gerhardt})
\[
\bar A u + \bar H = -  {\rm div} \Big(
\bar a (x, u(x), \nabla u(x))\Big) + \bar H.
\]
In our case the vector field $\bar a$ is given by $\frac{\partial L}{\partial p}$
and the forcing term $\bar H$ is given by $\frac{\partial L}{\partial z} (x, u(x), \nabla u(x))$.
In  \cite{Gerhardt} the forcing term $\bar H$ is assumed to be $W^{1,\infty}$
(see equation (5) in  \cite{Gerhardt}), which in the present situation would be verified only knowing
already that $u \in W^{2,\infty}$, which is however what one wants to deduce.
 
We do not exclude that going through the proofs of  \cite{Gerhardt} one could overcome such a difficulty; however we think that the  approach of the present paper could be
of independent interest, especially because it is self-contained and based
on an elementary use of Schauder estimates. 
\end{remark}
}

\subsection{Further comments}
We have proven the above regularity of the isoperimetric-isodiametric sets $E \subset M$ under the assumptions that the 
enclosing ball $B= B_{\rad(E)}(x_0)$ has smooth boundary. Actually, the following is true and is a direct consequence 
of the argument used above.

\begin{itemize}
\item[(A)] If $\de B \in C^{1,\alpha}$ for some $\alpha\in (0,1]$, then in a neighbourhood of $\de B$ the 
isoperimetric-isodiametric sets have the boundary $\de E$ which are $C^{1,\alpha}$ regular.
\end{itemize}

Indeed, under the assumption in (A), the arguments in Lemma~\ref{l:almost min} show that $\de E$ is 
$C^{1,\kappa}$ regular in a neighbourhood of $\de B$ for $k = \min\{\alpha, \frac12\}$. Moreover, a careful inspection 
of the proof of the optimal regularity in Theorem~\ref{p:regularity2} shows that the conclusion of (A) holds true 
with the right H\"older exponent (in the case $\alpha =1$ the proof is a straightforward generalization; for $\alpha\in 
(\frac12,1)$ more details need to be checked). Nevertheless, we do not do it here.

\end{document}